\documentclass[notitlepage,11pt,reqno]{amsart}
\usepackage{amssymb,amsmath,amscd,xspace}
\usepackage[mathscr]{eucal}
\usepackage[all]{xy}
\usepackage[breaklinks=true]{hyperref}%To get clickable links to papers

\DeclareMathOperator{\id}{\ensuremath{id}\xspace}

\DeclareMathOperator{\hhh}{\ensuremath{Hom}\xspace}
\DeclareMathOperator{\eee}{\ensuremath{Ext}\xspace}
\DeclareMathOperator{\sym}{\ensuremath{Sym}\xspace}
\DeclareMathOperator{\Ind}{\ensuremath{Ind}\xspace}
\DeclareMathOperator{\Res}{\ensuremath{{Res}_R^{R \rtimes \Gamma}\xspace}}
\DeclareMathOperator{\hgt}{\ensuremath{ht}\xspace}
\DeclareMathOperator{\ch}{\ensuremath{ch}\xspace}
\DeclareMathOperator{\chrc}{\ensuremath{char}\xspace}
\DeclareMathOperator{\ad}{\ensuremath{ad}\xspace}
\DeclareMathOperator{\Ad}{\ensuremath{Ad}\xspace}
\DeclareMathOperator{\Mod}{\ensuremath{Mod}\xspace}
\DeclareMathOperator{\End}{\ensuremath{End}\xspace}
\DeclareMathOperator{\Aut}{\ensuremath{Aut}\xspace}
\DeclareMathOperator{\wt}{\ensuremath{wt}\xspace}
\DeclareMathOperator{\tv}{\ensuremath{T_{\bf V}}\xspace}

\newcommand{\mf}[1]{\ensuremath{\mathfrak{#1}}\xspace}
\newcommand{\mfg}{\ensuremath{\mathfrak{g}}\xspace}
\newcommand{\spl}{\ensuremath{\mf{sl}_2}\xspace}
\newcommand{\usl}{\ensuremath{\mf{U}(\spl)}\xspace}
\newcommand{\usln}{\ensuremath{\mf{U}(\spl^{\oplus n})}\xspace}
\newcommand{\dd}{\ensuremath{\Delta}\xspace}
\newcommand{\la}{\ensuremath{\lambda}\xspace}
\newcommand{\aaa}{\ensuremath{\alpha}\xspace}
\newcommand{\nn}{\ensuremath{\mathbb{Z}_{\geqslant 0}}\xspace}
\newcommand{\N}{\ensuremath{\mathbb N}\xspace}
\newcommand{\Z}{\ensuremath{\mathbb Z}\xspace}
\newcommand{\C}{\ensuremath{\mathbb C}\xspace}
\newcommand{\calh}{\ensuremath{\mathcal{H}}\xspace}
\newcommand{\calo}{\ensuremath{\mathcal{O}}\xspace}
\newcommand{\caloo}{\ensuremath{\mathcal{O}_\N}\xspace}
\newcommand{\calc}{\ensuremath{\mathcal{C}}\xspace}
\newcommand{\cald}{\ensuremath{\mathcal{D}}\xspace}
\newcommand{\calp}{\ensuremath{\mathcal{P}}\xspace}

\newcommand{\Hom}{\ensuremath{\hhh_{A \rtimes \G}}\xspace}

\newcommand{\hhom}{\ensuremath{\hhh_\calo}\xspace}
\newcommand{\Ext}{\ensuremath{\eee_\calo}\xspace}
\newcommand{\bfg}{\ensuremath{(\Mod-B)^{fg}}\xspace}
\newcommand{\ag}{\ensuremath{A \rtimes \Gamma}\xspace}
\newcommand{\G}{\ensuremath{\Gamma}\xspace}
\newcommand{\g}{\ensuremath{\gamma}\xspace}
\newcommand{\gpm}{\ensuremath{\gamma^{\pm 1}}\xspace}
\newcommand{\hbg}{\ensuremath{(H \otimes B_+) \rtimes \Gamma}\xspace}
\newcommand{\sgx}{\ensuremath{S_\Gamma(x)}\xspace}
\newcommand{\gsl}{\ensuremath{\Gamma^{S^3_A(\lambda)}}\xspace}
\newcommand{\groto}{\ensuremath{{\rm Grot}_{\mathcal{O}}}\xspace}
\newcommand{\one}[1]{\ensuremath{{#1}_{(1)}}\xspace}
\newcommand{\two}[1]{\ensuremath{{#1}_{(2)}}\xspace}
\newcommand{\vi}{\ensuremath{\varepsilon}\xspace}
\newcommand{\tangle}[1]{\ensuremath{\langle #1 \rangle}\xspace}

\newcommand{\w}{\ensuremath{{\bf w}}\xspace}
\newcommand{\x}{\ensuremath{{\bf x}}\xspace}
\newcommand{\y}{\ensuremath{{\bf y}}\xspace}

\newcommand{\pfilt}{\ensuremath{\mathcal{F}(\dd)}\xspace}

\newcommand{\comment}[1]{}

\theoremstyle{plain}
\newtheorem{theorem}{Theorem}
\newtheorem{prop}{Proposition}
\newtheorem{lemma}{Lemma}
\newtheorem{cor}{Corollary}

\theoremstyle{definition}
\newtheorem{remark}{Remark}
\newtheorem{definition}{Definition}
\newtheorem{propdef}{Proposition-Definition}
\newtheorem{stand}{Standing Assumption}

\numberwithin{theorem}{section}
\numberwithin{equation}{section}
\numberwithin{lemma}{section}
\numberwithin{prop}{section}
\numberwithin{cor}{section}
\numberwithin{definition}{section}
\numberwithin{propdef}{section}
\numberwithin{stand}{section}
\numberwithin{remark}{section}

\begin{document}
\title{Functoriality of the BGG Category \calo}%over skew group rings}
\author{Apoorva Khare}
\date{\today}
\address{Department of Mathematics, Yale University, PO Box 208283, New
Haven, CT 06520, USA}
\dedicatory{Dedicated to the memory of Israel Moiseevich Gelfand}
\email{\small \tt apoorva.khare@yale.edu}
\subjclass[2000]{Primary: 16D90; Secondary: 16S35, 17B10}
\keywords{BGG Category $\calo$, (Hopf) RTA, skew group ring, block
decomposition}

\begin{abstract}
This article aims to contribute to the study of algebras with triangular
decomposition over a Hopf algebra, as well as the BGG Category $\calo$.
We study functorial properties of $\calo$ across various setups. The
first setup is over a skew group ring, involving a finite group $\G$
acting on a regular triangular algebra $A$. We develop Clifford theory
for $A \rtimes \G$, and obtain results on block decomposition, complete
reducibility, and enough projectives. $\calo$ is shown to be a highest
weight category when $A$ satisfies one of the ``Conditions (S)"; the BGG
Reciprocity formula is slightly different because the duality functor
need not preserve each simple module.

Next, we turn to tensor products of such skew group rings; such a product
is also a skew group ring. We are thus able to relate four different
types of Categories $\calo$; more precisely, we list several conditions,
each of which is equivalent in any one setup, to any other setup - and
which yield information about $\calo$.
\end{abstract}
\maketitle

\section{Introduction}

The results of this article relate the representation theories of various
algebras; thus, they are ``functorial" in a sense. However, one can apply
them to certain well-understood algebras, to get results in other setups.
For example, we show the following result in Proposition \ref{Pss} and
after Remark \ref{R1}; for details and more results in this setting, also
look after Remark \ref{R1}.

\begin{prop}\label{Pexample}
Given a complex semisimple Lie algebra $\mfg$, denote by $P^+_{\mfg}$ its
set of dominant integral weights, and $R_{\mfg}$ the {\em wreath product
algebra} $S_n \wr \mf{Ug}$ (see \cite[\S 6]{Mac}). Then the category of
finite-dimensional $R_{\mfg}$-modules contains at least
``$P^+_{\mfg}$-many" simple objects, and is semisimple.
\end{prop}

\noindent Hereafter, $S_n \wr A := A^{\otimes n} \rtimes S_n$ for any
ring $A$, with $S_n$ acting by permuting the components in the tensor
product; the definition is similar when $A$ is a group.\medskip

This article studies the representation theory of special families of
algebras with triangular decomposition over a commutative Hopf algebra
(these algebras have been studied in general by Bazlov and Berenstein).
In general, such algebras are not Hopf algebras; thus one studies them,
for example, by defining and analyzing (analogues of) Verma modules - or,
in other words, some version of the Bernstein-Gelfand-Gelfand Category
$\calo$. We do so below, for a special subclass of such algebras.

In \cite{Kh3}, we defined a general framework of a {\em regular
triangular algebra} $A$ (also denoted by {\em RTA}; we recall the
definition below), wherein the BGG Category $\calo$ can be studied, and
results obtained about a block decomposition into highest weight
categories. Examples of such algebras are (quantized) universal
enveloping algebras of (semisimple, or) symmetrizable Kac-Moody Lie
algebras, Heisenberg and Virasoro algebras, and (quantized) infinitesimal
Hecke algebras (see \cite{Kh,GGK,EGG}, and \cite{KT,Kh3} respectively).

The goal of this article is to extend many of the classical results of
\cite{BGG1} to other setups (e.g., Proposition \ref{Pexample} above), by
applying the following two constructions:
\begin{itemize}
\item the {\em skew group ring} $A \rtimes \G$ (where $\G$ is a
finite group acting on $A$),

\item the {\em tensor product} of RTAs $A_1, \dots, A_n$ for some $n$.
\end{itemize}

\noindent These constructions were motivated by the study of
finite-dimensional representations of the wreath product symplectic
reflection algebra in \cite{EM}; we term the first construction {\em
Clifford theory}. Combining them produces results for the wreath product
of an RTA, for instance.\medskip

We now combine the two constructions as follows: suppose $A_i \rtimes
\G_i$ are skew group rings for $1 \leq i \leq n$. We can then form $A =
\otimes_{i=1}^n A_i, \G = \times_{i=1}^n \G_i$, and $A \rtimes \G$ - and
this gives us four different setups for the category $\calo$:
\begin{equation}\label{E1}
\mbox{all } A_i = \{ A_i : 1 \leq i \leq n \}, \mbox{ all } A_i \rtimes
\G_i = \{ A_i \rtimes \G_i : 1 \leq i \leq n \}, A, A \rtimes \G.
\end{equation}

\noindent In a sense, these constructions ``commute" when $\G = \times_i
\G_i$, namely:
\begin{equation}\label{Ediag1}
\begin{CD}
\{ A_i \} @>\rtimes>> \{ A_i \rtimes \G_i \}\\
@V\otimes VV @V\otimes VV\\
A = \otimes_i A_i @>\rtimes>> A \rtimes \G
\end{CD}
\end{equation}

\noindent (Every diagram in this article is functorial, rather than a
commuting square of morphisms in some categories.)
However, if we want to construct the wreath product $S_n \wr A$, say,
then the diagram above does not help: the steps to take are $\{ A_i = A,
|\G_i| = 1 \} \longrightarrow A^{\otimes n} \longrightarrow A^{\otimes n}
\rtimes S_n$. But these steps are all found in diagram \eqref{Ediag1};
moreover, all algebras here (as well as in \eqref{Ediag1}) are examples
of skew group rings.

Our goal, therefore, is twofold: (a) to relate the categories $\calo$ in
the above four setups, and (b) to show that $\calo_{\ag}$ is a highest
weight category in the sense of Cline, Parshall, and Scott \cite{CPS1},
when $A \rtimes \G$ (or $A$) satisfies what was called {\em Condition
(S)} in \cite{Kh} - but we now call {\em Condition (S3)}, as in
\cite{Kh3}. As we will see, this is related to this condition being
satisfied in the other three setups.

Getting results using the second vertical arrow in diagram \eqref{Ediag1}
may pose problems - for instance, see diagram \eqref{Ediag2} below.
However, the horizontal arrows can be ``reversed", which allows us to
proceed the ``longer" way in this case.
Also note that much of the analysis will be analogous to the theory
developed in \cite{BGG1,Kh,GGK}; however, we will explain the new
features - as well as the interconnections - in detail.\medskip

If $H$ is cocommutative as well, then skew group rings are algebras with
triangular decomposition over the Hopf algebra $H \rtimes \G$ (see
\cite[Appendix]{BaBe}). Thus, one avenue for possible further study, is
to bring the theory of braided doubles and triangular ideals into the
picture.

Finally, as a small application, we study the wreath product of
symplectic oscillator algebras; we conclude by showing that the
Poincare-Birkhoff-Witt property does not hold if we deform certain
relations.

\section{Setups - the general and Hopf cases}

We work throughout over a ground field $k$. Unless otherwise specified,
all tensor products are over $k$. Define $\nn := \N \cup \{ 0 \}$. Given
$S \subset \Z$ and a finite subset $\dd$ of an abelian group $\calp_0$,
the symbols $(\pm) S \dd$ stand for $\{ (\pm) \sum_{\aaa \in \dd} n_\aaa
\aaa : n_\aaa \in S\ \forall \aaa \} \subset \calp_0$. We will often
abuse notation and claim that two modules or functors are equal, when
they are isomorphic (double duals, for instance). Finally, in developing
Clifford theory for RTAs and the Category $\calo$, we use the general
results on Clifford theory over $\C$, that are stated in the Appendix in
\cite{Mac}.

\begin{definition}\hfill
\begin{enumerate}
\item If $A$ is a $k$-algebra, and $\G$ a group acting on $A$ by
$k$-algebra automorphisms, then the {\em ($\G$-)skew group ring over $A$}
is defined to be $A \rtimes \G := A \otimes_k k\G$, with relations $\g a
= \g(a) \g$. (Henceforth, $k\G$ is the group algebra of $\G$.)

\item For $\g \in \G$, define $\Ad \g \in \Aut_k (A \rtimes \G)$ to be
$\Ad \g(a \g') = \g a \g' \g^{-1} = \g(a) \Ad_\G \g(\g')$.

\item Given a {\em weight} $\la \in \hhh_{k-alg}(A,k)$ and an $A$-module
$M$, its {\em $\la$-weight space} is $M_\la := \{ m \in M : am = \la(a)
m\ \forall a \in A \}$.
\end{enumerate}
\end{definition}

\subsection{The main definitions}

We now mention the general setup under which our results are proved; for
``all practical purposes", the assumptions are simpler, and to see them,
the reader should go ahead directly to \S \ref{Sshopf}.

\begin{definition}
An associative $k$-algebra $A$, together with the following data, is
called a {\em regular triangular algebra} (denoted also by {\em RTA}; see
\cite{Kh3}).
\begin{enumerate}
\item[(RTA1)] There exist associative unital $k$-subalgebras $B_\pm$ and
$H$ of $A$, such that the multiplication map $: B_- \otimes_k H \otimes_k
B_+ \to A$ is a vector space isomorphism (the {\em triangular
decomposition}).

\item[(RTA2)] There is an algebra map $\ad \in \hhh_{k-alg}(H,
\End_k(A))$, such that for all $h \in H,\ \ad h$ preserves each of $H,
B_\pm$ (identifying them with their respective images in $A$). Moreover,
$H \otimes B_\pm$ are $k$-subalgebras of $A$.

\item[(RTA3)] There exists a free action $*$ of a group $\calp$ on $G :=
\hhh_{k-alg}(H,k)$, as well as a distinguished element $0_G = 0_\calp *
0_G \in G$ such that $H = H_{0_G}$ as an $\ad H$-module.

\item[(RTA4)] There exists a subalgebra $H_0$ of $H$, and a free abelian
group $\calp_0$ of finite rank, such that
\begin{enumerate}
\item $\calp_0$ acts freely on $G_0 := \hhh_{k-alg}(H_0,k)$ (call this
action $*$ as well), and

\item the ``restriction" map $\pi : G \to G_0$ sends $\calp * 0_G$ onto
$\calp_0 * \pi(0_G)$, and intertwines the actions, i.e., $* \circ (\pi
\times \pi) = \pi \circ *$.
\end{enumerate}\medskip

For the remaining axioms, we need some notation. Fix a finite basis $\dd$
of $\calp_0$. For each $\theta \in \calp$ and $\theta_0 \in \calp_0 = \Z
\dd$, abuse notation and define
$\theta = \theta * 0_G \in G,\ \theta_0 = \theta_0 * \pi(0_G) \in G_0$.
(We will differentiate between $0 \in \calp_0$ or $G_0$, and $0_G \in
G$.) We call $G$ (or $G_0, \calp_0, \dd$) the set of {\em weights} (or
the {\em restricted weights, root lattice, simple roots} respectively).

Given $\la \in S \subset G$ and a module $M$ over $H$ (e.g., $M = (A,
\ad)$), define the {\em weight space} $M_\la$ as above, and $M_S :=
\bigoplus_{\la \in S} M_\la$. Given $\theta_0 \in \Z \dd$, define
$M_{\theta_0} := M_{\pi^{-1}(\theta_0)}$.\medskip

\item[(RTA5)] It is possible to choose $\dd$, such that
$\displaystyle B_\pm = \bigoplus_{\theta \in \calp : \pi(\theta) \in \pm
\nn \dd} (B_\pm)_\theta$
(where $A$ is an $H$-module via $\ad$).

\item[(RTA6)] $(B_\pm)_0 = (B_\pm)_{\pi^{-1}(\pi(0_G))} = k$, and
$\dim_k(B_\pm)_{\theta_0} < \infty\ \forall \theta_0 \in \pm \nn \dd$ (we
call this {\em regularity}).

\item[(RTA7)] The {\em property of weights} holds: for all $A$-modules
$M$,
\begin{eqnarray*}
A_\theta \cdot A_{\theta'} & \subset & A_{\theta * \theta'}\ \forall
\theta, \theta' \in \calp,\\
A_\theta \cdot M_\la & \subset & M_{\theta * \la}\ \forall \theta \in
\calp, \la \in G.
\end{eqnarray*}

\item[(RTA8)] There exists an anti-involution $i$ of $A$ (i.e., $i^2|_A =
\id|_A$) that acts as the identity on all of $H$, and takes $A_\theta$ to
$A_{\theta^{-1}}$ for each $\theta \in \calp$.
\end{enumerate}
\end{definition}

\begin{definition}
An RTA is {\em strict} if $H = H_0, G = G_0 \supset \calp = \calp_0$
(whence $\pi = \id|_G$).
\end{definition}

\noindent {\bf Example.} This definition is quite technical; here is our
motivating example - a complex semisimple Lie algebra $\mf{g}$. Then $A =
\mf{Ug},\ \ad$ is the standard adjoint action, and $H = H_0 = \sym
\mf{h}$, whence the set of weights is $G = G_0 = \mf{h}^* \supset \calp =
\calp_0 = \Z \dd$ (the root lattice). Moreover, $i$ is the composite of
the Chevalley involution and the Hopf algebra antipode on $\mf{Ug}$.

\begin{remark}
We note that $H$ is commutative by (RTA8), and (RTA6) defines
augmentation ideals $N_\pm$ of $B_\pm$. One change from earlier theories
of the Category $\calo$, is in our allowing ``non-strict" RTAs in our
setup; this is needed if we want to include infinitesimal Hecke algebras
(not over $\mf{sl}_2$). See \cite{Kh3,KT} for more details.
\end{remark}

\begin{stand}\label{St1}
Henceforth, $A$ is an RTA, $\G$ is a finite group acting on $A$, and $k$
is a field of characteristic zero if $|\G| > 1$. Moreover, in $A \rtimes
\G$,
\begin{enumerate}
\item There is an algebra map $\ad : H \rtimes \G \to \End_k(A)$, that
restricts on $H, \G$ to $\ad \in \hhh_{k-alg}(H, \End_k(A))$ and $\Ad \in
\hhh_{group}(\G, \Aut_{k-alg}(A))$ respectively. Moreover, $i(\g(a)) =
\g(i(a))$ for $\g \in \G, a \in A$, and each subalgebra $R = k, N_\pm,
H_0, H$ is preserved by $\ad \xi$, for each $\xi \in H \rtimes \G$. Here,
the restricted $\ad|_H$ and the anti-involution $i$ are part of the
RTA-structure of $A$.

\item The map $: \G \times G \to G = \hhh_{k-alg}(H,k)$, given by
$\tangle{\g(\la), h} := \tangle{\la, \g^{-1}(h)}$, is an action that
preserves $(\calp * 0_G, *)$. That is, $\g(\theta) * \g(\la) = \g(\theta
* \la)$ for all $\g, \theta, \la$ respectively in $\G, \calp, G$, and $\g
: \calp * 0_G \to \calp * 0_G\ \forall \g$.
\end{enumerate}
\end{stand}

The above assumptions imply the following ``compatibility":

\begin{lemma}\label{Lcompat}
Suppose $A \rtimes \G$ is a skew group ring that satisfies Assumption
\ref{St1}. Then $\G$ preserves $0_G$, and also acts on $G_0$ (and
$\calp_0$), such that $\pi$ intertwines the actions: $\forall \g \in \G,\
\pi \circ \g = \g \circ \pi$ on $G$. Moreover, $\g(A_\theta) =
A_{\g(\theta)}\ \forall \g \in \G, \theta \in \calp$, and $\ag$ has an
anti-involution that restricts to $i, i_{\G}$ on $A,\G$ respectively.
\end{lemma}

To show this, we need some basic results.

\begin{lemma}\label{L1}
Suppose a group $\G$ acts on an associative unital $k$-algebra $R$ by
$k$-algebra automorphisms. Then $\G$ acts on $R^*$:
$\tangle{\g(\la),r} := \tangle{\la, \g^{-1}(r)}$.
\begin{enumerate}
\item Given $\g \in \G$, an $R$-module $M$, and a weight $\la$,
$\g(M_\la) = M_{\g(\la)}$.

\item Suppose $i : R \to R$ is an anti-involution. Define $i_\G(\g) =
\g^{-1}$ for $\g \in \G$. Then $i,i_\G$ extend to an anti-involution of
$R \rtimes \G$, if and only if $i(\g(r)) = \g(i(r))$ for all $r \in R, \g
\in \G$.
\end{enumerate}
\end{lemma}

\begin{proof}[Proof of Lemma \ref{Lcompat}]
That $\g (A_\theta) = A_{\g(\theta)}$ follows from the above lemma.
Moreover, $H = \g(H) = \g(H_{0_G}) = H_{\g(0_G)}$, whence $\g$ fixes
$0_G$. Finally, given $\la \in G$, one easily checks that $\g(\pi(\la)) =
\pi(\g(\la))$ on $H_0$; in turn, this implies that $\G$ acts on $\calp_0$
(since it acts on $\calp$ by Assumption \ref{St1}).
\end{proof}

\noindent We used the (standard) Hopf algebra structure of $k\G$ in the
above results (i.e., $\dd(\g) = \g \otimes \g, S(\g) = \g^{-1}, \vi(\g) =
1$). This is further used in the following result, which also helps us
rephrase (and reduce) the assumptions when $H \supset H_0$ are Hopf
algebras (i.e., $A$ is a {\em Hopf RTA}; see \cite{Kh3}).

\begin{lemma}
Keep the assumptions of Lemma \ref{L1}, and suppose also that $R$ is a
Hopf algebra. We thus have Hopf algebra operations $\dd, \vi,S$ on both
the (Hopf) subalgebras $R,k\G$ of $R \rtimes \G$.
\begin{enumerate}
\item These operations on $R,k\G$ extend to $R \rtimes \G$ (such that $R
\rtimes \G$ becomes a Hopf algebra), if and only if $\Ad$ is a group
homomorphism from $\G$ to {\em Hopf algebra} automorphisms $\Aut_{Hopf}
(R)$.
\end{enumerate}

\noindent Now suppose that the conditions in the first part hold.
\begin{enumerate}
\setcounter{enumi}{1}
\item Then $\Ad \in \hhh_{group}(\G, \Aut_{group}(G))$, where $G =
\hhh_{k-alg}(R,k)$ is a group under \emph{convolution}: $(\la * \mu)(r)
:= \sum \la(\one{r}) \mu(\two{r})$.

\item If $A \supset R$ is a $k$-algebra containing $R$ (with $1_A =
1_R$), so that $\G$ acts on $A$ by algebra maps (with compatible
restriction to $R$), then $\ad|_R$ and $\Ad|_{\G}$ can be extended to
$\ad \in \hhh_{k-alg}(R \rtimes \G, \End_k (A \rtimes \G))$.
\end{enumerate}
\end{lemma}

We remark that not every algebra automorphism of a Hopf algebra is a Hopf
algebra automorphism; for example, if $H$ is a non-cocommutative Hopf
algebra, then the flip map $\tau : H \otimes H^{cop} \to H \otimes
H^{cop}$, given by $\tau(x \otimes y) = y \otimes x$, is an algebra map
but not a coalgebra map. Here, $H^{cop}$ denotes the Hopf algebra
$(H,m,\eta,\dd^{op},\vi,S^{-1})$.

\subsection{The case of Hopf algebras}\label{Sshopf}

Several extensively studied examples in representation theory occur with
the additional data that $H \supset H_0$ are Hopf algebras (see
\cite{Kh3} for a theorem, as well as a list of examples). Thus, this is
the setup one should have in mind.

The lemmas above, together with the analysis in \cite{Kh3}, show that
some of the defining assumptions can be relaxed. (In particular, $A$ is
now a {\em Hopf RTA}, or {\em HRTA}, when $\G$ is trivial.) Let us
mention the ``reduced" set of axioms for $A$ and $\ag$, obtained by
combining all this.

\begin{propdef}
A {\em skew group ring over a Hopf RTA} is $A \rtimes \G$, where all but
the last part ($\G(\dd) = \dd$) hold if and only if (see \cite{Kh3}) $A$
is an RTA, $H \supset H_0$ are Hopf algebras with compatible structures,
and $\ad$ their usual adjoint actions.
\begin{enumerate}
\item The multiplication map: $B_- \otimes_k H \otimes_k B_+ \to A$ is a
vector space isomorphism. Here $H,B_\pm$ are associative unital
$k$-subalgebras of $A$. Moreover, the ``Cartan part" $H$ is a commutative
Hopf algebra.

\item $H$ contains a sub-Hopf algebra $H_0$ (with groups of weights
$G,G_0$ respectively), and $G_0$ contains a free abelian group of finite
rank $\calp_0 = \Z \dd$. Here, $\dd$ is a basis of $\calp_0$, chosen such
that
\[ B_\pm = \bigoplus_{\theta \in G : \pi(\theta) \in \pm \nn \dd}
(B_\pm)_\theta = \bigoplus_{\theta_0 \in \pm \nn \dd} (B_\pm)_{\theta_0}
\]

\noindent (the summands are weight spaces under the usual adjoint
actions). Each summand in the second sum is finite-dimensional, and
$(B_\pm)_{0_G}$ $= (B_\pm)_0 = k$.

\item There exists an anti-involution $i$ of $A$, such that
$i|_H = \id|_H$.

\item $\G$ is a finite group, that acts on $B_\pm, H, H_0$ (and hence on
$A$), such that the action of each $\g \in \G$
\begin{itemize}
\item on $B_\pm$ is by algebra automorphisms,

\item on $H$ (and hence on $H_0$) is by Hopf algebra automorphisms,

\item on $A$ commutes with the anti-involution $i$, and

\item induced on $G_0$ preserves $\dd$.
\end{itemize}
\end{enumerate}
\end{propdef}

\begin{remark}
First, we do not assume here, that $H$ is cocommutative; nevertheless, $H
\otimes B_\pm \cong B_\pm \rtimes H$, the smash product algebras. Next,
that $\G$ preserves $\dd$ is included, because it will be needed later to
show that every Verma module has a unique simple quotient; see
Proposition \ref{P6} below.
\end{remark}

Recently, Bazlov and Berenstein defined a class of algebras that
encompasses symmetrizable Kac-Moody Lie algebras and their quantum
groups, and rational Cherednik algebras. We now mention their connection
to skew group rings (so one may now try to apply their results in this
setup).

\begin{definition}
A $k$-algebra $A$ has {\em triangular decomposition over a bialgebra}
$H'$, if $A$ has distinguished subalgebras $H', U^\pm$ such that
\begin{itemize}
\item $H'$ acts covariantly from the left on $U^-$, and from the right on
$U^+$;

\item the multiplication map $: U^- \otimes H' \otimes U^+ \to A$ is a
vector space isomorphism, that makes $U^- \otimes H'$ and $H' \otimes
U^+$ isomorphic to the smash products $U^- \rtimes H'$ and $H' \ltimes
U^+$ (by the above actions) respectively;

\item there exist $H'$-equivariant (via the counit $\vi$) characters
$\epsilon^\pm : U^\pm \to k$.
\end{itemize}
\end{definition}

\begin{prop}\label{Pbabe}
Skew group rings over HRTAs are examples of algebras with triangular
decomposition over $H \rtimes \G$.
\end{prop}

\begin{proof}
Set $U^\pm = B_\pm$ and $H' = H \rtimes \G$. Moreover, define the two
actions of $H \rtimes \G$ (on all of $A$, in fact) to be: $h
\triangleright a := \ad h(a),\ a \triangleleft h := \ad S(h)(a)$. Since
$B_\pm$ are direct sums of $\ad H$-weight spaces, and closed under $\ad
\G$, hence one can check that these are valid left and right $H \rtimes
\G$-actions (note that $S^2 = \id|_{H'}$, since $H$ is commutative and
$\G$ is cocommutative). It is now easy to verify the first two
conditions.
Finally, define the characters $\epsilon^\pm$ to have augmentations
$N_{\pm}$. Over here, since $H$ is a Hopf algebra, we have $0 = \vi : H
\to k$ (and $\vi$ can be extended to $H \rtimes \G$). It is now easy to
verify that $\epsilon^\pm$ are $H'$-equivariant (via $\vi$), for we
verify on each $\theta$-weight space, that
$\epsilon^\pm(\ad(h \g)(u_\theta^\pm)) = \delta_{\g(\theta),\vi} \vi(h)
u_{\g(\theta)}^\pm = \vi(h\g) \epsilon^\pm(u_\theta^\pm)$.
\end{proof}

\subsection{Examples}\hfill
\begin{enumerate}
\item The degenerate example is that of an RTA $A$, where we take $\G =
1$.

\item The {\em wreath product} $S_n \wr A$ is defined to be $A^{\otimes
n} \rtimes S_n = S_n \wr A$, where $S_n$ is the group of permutations of
$\{ 1, \dots, n\}$, and $A$ is a (strict) (Hopf) RTA. By \cite{Kh3},
$A^{\otimes n}$ is also a (strict) (Hopf) RTA, with simple roots $\dd =
\coprod_i \dd_i$, and weights $G = G_A^n, \calp = \calp_A^n$, and so on.

Define $f_i : A \hookrightarrow A^{\otimes n} \subset S_n \wr A$, sending
$a$ to the product of $1^{\otimes (i-1)} \otimes a \otimes 1^{\otimes
(n-i)} \otimes 1_{\G}$; then the relations are $s_{ij} f_i(a) = f_j(a)
s_{ij}$, where $a \in A$, and $s_{ij}$ is the transposition that
exchanges $i$ and $j$. Thus, $\sigma(\calp_i) = \calp_{\sigma(i)}$, and
$\sigma(\theta * \la) = \sigma(\theta) * \sigma(\la)\ \forall \theta \in
\calp_A^n, \la \in G_A^n$.

Moreover, $\sigma(\aaa_i) = \aaa_{\sigma(i)}$, where $\sigma \in S_n$,
$\aaa \in \dd$ is a simple root, and $\aaa_i := f_i^*(\aaa) : (H_0)_i \to
k$. (So $\sigma(\dd_i) = \dd_{\sigma(i)}$.)
Finally, define
\[ \ad(h_1 \otimes \dots \otimes h_n \otimes \sigma)(a_1 \otimes \dots
\otimes a_n) := \otimes_j \ad h_j(a_{\sigma^{-1}(j)}). \]

\noindent One checks that each $\Ad \sigma$ acts by a Hopf algebra
automorphism if $H$ (and hence $H^{\otimes n}$) is a Hopf algebra. Thus,
$S_n \wr A$ satisfies all the standing assumptions.

\item If $A \rtimes \G$ is a skew group ring over a (strict) (Hopf) RTA,
then so is $A \rtimes \G'$, for any subgroup $\G'$ of $\G$.

\item For any finite $\G$, $A \otimes k\G$ is a skew group ring if
$A^{\G} = A$.

\item If $A_i \rtimes \G_i$ are skew group rings that satisfy the above
assumptions, then we know by \cite{Kh3}, that $A = \otimes_i A_i$ is an
RTA, the set $G$ of weights is $\times_i G_i$, and $\G = \times_i \G_i$
acts on $A$, via:
\[ (\g_1, \dots, \g_n) \cdot (a_1 \otimes \dots \otimes a_n) = (\g_1(a_1)
\otimes \dots \otimes \g_n(a_n)). \]

\noindent One can check that $A \rtimes \G$ also satisfies all the
assumptions above.

Finally, if $H_i$ is a Hopf algebra for all $i$, then so is $H =
\otimes_i H_i$. If each $\Ad \g_i$ acts as a Hopf algebra automorphism on
$A_i$, then the same property holds for $\Ad \G$ acting on $A$.
\end{enumerate}

\section{The Bernstein-Gelfand-Gelfand Category}

We now introduce the main object of study in this article.

\begin{definition}
The {\em BGG category} $\calo$ is the full subcategory of finitely
generated $H$-semisimple $A \rtimes \G$-modules, with finite-dimensional
weight spaces and a locally finite action of $B_+$, i.e., $\forall m \in
M \in \calo,\ \dim(B_+ m) < \infty$.
\end{definition}\medskip

\noindent Then $\calo$ is closed under quotienting, and every object $M$
in $\calo$ has a locally finite action of $(H \otimes B_+ \otimes k\G)$.
%
%
%%%%%%
%This is because $(B_+ \otimes H \otimes k\G)m = \sum_{\g \in \G} (B_+
%\otimes H) (\g m)$, and this is finite-dimensional since \G is finite.
%%%%%%
%
%
Moreover, viewing each $M \in \calo$ as merely an $A$-module, we
can show the following lemma, since $\G$ is finite.

\begin{lemma}\label{LO}
Suppose $\calo_A$ is the Category $\calo$ for $|\G| = 1$. Then $\calo =
\calo_{A \rtimes \G}$ equals $A \rtimes \G - \Mod \bigcap \calo_A := \{ M
\in A \rtimes \G - \Mod : {\rm Res}^{A \rtimes \G}_A M \in \calo_A \}$.
\end{lemma}

Let us first show that complete reducibility for finite-dimensional
$A$-modules implies it for $A \rtimes \G$-modules; the special case $A =
(\mf{Ug})^{\otimes n}, \G = S_n$ was stated in Proposition \ref{Pexample}
above (but one can also state that result for $U_q(\mf{g})$, for
instance). We use the following general homological result.

\begin{prop}\label{Pss}
Given a finite group $\G$ acting on an algebra $R$ over a field of
characteristic zero, suppose $\calp \subset R-\Mod$ and $\cald \subset (R
\rtimes \G)-\Mod$ are full abelian subcategories of finite-dimensional
modules, with each $D \in \cald$ satisfying: $\Res D \in \calp$. If
$\calp$ is a semisimple category, then so is $\cald$.
\end{prop}

\noindent The second part of Proposition \ref{Pexample} now follows, by
setting $\calp,\cald$ to be the categories of finite dimensional $A, A
\rtimes \G$- modules respectively. The first part will be shown after
Remark \ref{R1}.

\begin{proof}
It suffices to show that $\hhh_{\cald}(M,-)$ is exact for each object $M$
of $\cald$ (since $\cald$ is abelian, the long exact sequence of
Ext$_{\cald}$'s vanishes). Since $\cald$ is also full, we use
\cite[Equation (A.1), Appendix]{Mac}, and compute:
\[ \hhh_{\cald}(M,-) = \hhh_{R \rtimes \G}(M,-) = \left( \hhh_R(\Res M,
\Res -) \right)^{\G} \]

\noindent By Maschke's Theorem (over characteristic zero), taking
$\G$-invariants is an exact functor (since everything is
finite-dimensional here), as is $\Res - : \cald \to \calp$. By
semisimplicity in the full subcategory $\calp$, $\hhh_{\calp}(\Res M,-)$
is also exact. Thus their composite is exact as well.
\end{proof}

\section{Summary of results}

We now summarize our main results.

\begin{stand}\label{St3}
Suppose $A_i \rtimes \G_i$ are skew group rings over RTAs for $1 \leq i
\leq n$, each of which satisfies Standing Assumption \ref{St1} above.
Suppose also that each $\G_i$ preserves $\dd_i \subset (G_0)_i$, and that
$k$ is algebraically closed if any $\G_i$ is nontrivial.
\end{stand}\medskip

\noindent Then all this also holds for $A \rtimes \G$, where $A :=
\otimes_i A_i$ and $\G := \times_i \G_i$.
Hence, we will state results only for $\ag$ wherever possible,
because it combines all the functorial constructions above (see
\eqref{E1}). One can thus use $A$ as an RTA or as $\otimes_i A_i$. To see
the results for any of the ``subcases", take $n=1$ (to study only
Clifford theory), $\G$ to be trivial (to study only RTAs), etc.\medskip

Let $\calc$ be the category of finite-dimensional $H$-semisimple $H
\rtimes \G$-modules, and $X$ the isomorphism classes of simple objects in
$\calc$ (so $X = G$ if $|\G| = 1$). Then
\begin{itemize}
\item $\calc$ is semisimple.

\item $X$ classifies the simple objects in $\calo = \calo_{A \rtimes
\G}$, and for each $x \in X$, there is a {\em Verma module} $Z(x) \in
\calo$ with unique simple quotient $V(x)$.

\item Given $x$, there is a $\G$-orbit $\la_x \in G / \G$ (its ``set of
weights"), such that
\begin{itemize}
  \item for each $\la \in G$, there exists at least one - and only
  finitely many - $x \in X$ such that $\la \in \la_x$.

  \item (``Weyl Character Formula 1".) $V(x) \cong \bigoplus_{\la \in
  \la_x} V_A(\la)^{\oplus \dim x_\la}$ as $A$-modules in $\calo_A$, where
  $V_A(\la)$ is simple in $\calo_A$.

  \item The center $Z := \mf{Z}(\ag)$ acts on a Verma or simple module by
  a {\em central character} $\chi_x : Z \to k$.
\end{itemize}
\end{itemize}

Next, we define duality functors $F$ on the {\em Harish-Chandra
categories} $\calh$ containing $\calo$ and $\calc$, using the
anti-involution $i$ on $A$ and $H$ respectively.
\begin{itemize}
\item $F$ is exact, contravariant, and involutive on $\calc$ and $\calh$.
\item $F(V(x)) = V(F(x))$.
\item Simple modules $V(x), V(x')$ have non-split extensions in $\calo$
if and only if they or $V(F(x')),V(F(x))$ are the first two Jordan-Holder
factors in a composition series for a Verma module.
\end{itemize}

\noindent For example, if $|\G| = 1$, then $X = G$ and $F(V(\la)) \cong
V(\la)\ \forall \la \in G$.

\begin{definition}\label{D1}
Fix $x \in X$.
\begin{enumerate}
\item Define $CC(x) = S^4(x)$ to be $\{ x' \in X : \chi_{x'} = \chi_x
\}$.

\item Define $S^3(x)$ to be the symmetric and transitive (or equivalence)
closure of $\{ x \}$ in $X$, under the relations (a) $x \to x'$ if $V(x)$
is a subquotient of $Z(x')$, and (b) $x \to_F x'$ if $F(x) = x'$.

\item Define $S'(x)$ to be the equivalence closure of $\{ x \}$ in $X$
under only the relation (a) above.

\item Define the following partial order on $X$: $x \leq x'$ if $x=x'$,
or there exist $\la \in \la_x, \la' \in \la_{x'}$, with $\pi(\la') \in
(\nn) * \pi(\la)$.

\item Given $S \subset X \ni x$, define $S^{\leq x} := \{ s \in S : s
\leq x \}$ (and as a special case, $S^{\leq \la}$ if $S \subset G$ and
$|\G| = 1$). Similarly define $S^{<x}$ and $S^{<\la}$.

\item $S^2(x) := \{ \pi(\la_{x'}) : x' \in S^3(x) \},\ S^1(x) := \{
\pi(\la_{x'}) : x' \in S^3(x)^{\leq x} \}$.

\item $A \rtimes \G$ satisfies {\em Condition (S1),(S2),(S3)}, or {\em
(S4)} if the corresponding sets $S^n(x)$ are finite for all $x \in X$.

\item The {\em block} $\calo(x)$ is the full subcategory of $\calo$,
comprising all objects in $\calo$, each of whose simple subquotients is
in $\{ V(x') : x' \in S^3(x) \}$.

\item Given skew group rings $A_i \rtimes \G_i$ over RTAs for $1 \leq i
\leq n$, each of which satisfies Standing Assumption \ref{St3} above,
define $X_i$ similar to $X$, but over $H_i \rtimes \G_i$ for all $i$.
Given $\la_i \in G_i\ \forall i$, the {\em simple objects over} $\la =
(\la_1, \dots, \la_n)$ in the four setups in \eqref{E1} are,
respectively,
\begin{equation}\label{Esimples}
\la_i, \{ x_i \in X_i : \la_i \in \la_{x_i} \}, \la, \{ \x \in X : \la
\in \la_{\x} \}.
\end{equation}
\end{enumerate}
\end{definition}

\begin{remark}\hfill
\begin{enumerate}
\item As we see below, $X = \times_i X_i$ here, so $\la_{\x} \in \times_i
G_i = G$, as it should.

\item That Verma modules $Z(x)$ are indecomposable and have a unique
simple quotient $V(x)$, is related to the fact that $Z(x)$ is the
projective cover of $V(x)$ in a ``truncated" subcategory of $\calo$. (See
Lemma \ref{Lproj}.)

\item The relation between simple objects in the four setups for Category
$\calc$ (and $\calo$) is expected from diagram \eqref{Ediag1}, and
``indicated" in \eqref{Esimples}. Its analogue in $\calo$ forms the
``front face" of the ``cube of simple objects" below.
We show later, that $V(\x) := \otimes_i V_i(x_i) \in \calo_{A \rtimes
\G}$. We now state all this using diagrams; in them, ``$\rtimes$" really
is some sort of induction, or an ``inverse" to ``restriction" (from $A_i
\rtimes \G_i$ to $A_i$).

\item The Conditions (S) are not uncommon: (S3) holds for complex
semisimple Lie algebras, (nontrivially deformed) infinitesimal Hecke
algebras, and both their quantum analogues (e.g., see the example after
Remark \ref{R1}, \cite{Kh,Kh3,GGK} respectively).
\end{enumerate}
\end{remark}

\begin{theorem}\label{Teqs}
Fix $1 \leq m \leq 4$, and work in the above setup.
\begin{enumerate}
\item $X = \times_i X_i$.
\item Inside any such $\ag$, $\g(S^m_A(\la)) = S^m_A(\g(\la))\ \forall
\la \in G, \g \in \G$.

\item Given RTAs $A_i$ and $\la_i \in G_i$, $S^m_A(\la_1, \dots, \la_n) =
\times_i S^m_{A_i}(\la_i)$.
\end{enumerate}
\end{theorem}

\begin{prop}\label{Pdiag}\hfill
\begin{enumerate}
\item The following diagram of ``posets of simple objects in $\calo$"
commutes (in the spirit of diagram \eqref{Ediag1}) in between various
$\calc$'s:
\[ \begin{CD}
\{ G_i \} @>\rtimes>> \{ X_i \}\\
@V\times VV @V\times VV\\
G = \times_i G_i @>\rtimes>> X = \times_i X_i
\end{CD} \]

\item Moreover, given $\la_i \in G_i$ and $x_i \in X_i$ (for all $i$)
such that $\la_i \in \la_{x_i}$, the three constructions of $\otimes_i$,
$F$, and ``$\rtimes$" form a commuting ``cube of simple objects" in
various $\calo$'s:
\[
\xymatrix{
  & \{ V_i(\la_i) \}
     \ar[rr]^{\rtimes}
     \ar[dd]^{\otimes}
  && \{ V_i(F(x_i)) \}
     \ar[dd]^{\otimes}\\
  \{ V_i(\la_i) \}
     \ar[rr]^{\rtimes}
     \ar[dd]^{\otimes}
     \ar[ru]^{F}
  && \{ V_i(x_i) \}
     \ar[dd]^{\otimes}
     \ar[ru]^{F}\\
  & V(\la)
     \ar[rr]^{\rtimes}
  && V(F(\x))\\
  V(\la)
     \ar[rr]^{\rtimes}
     \ar[ru]^{F}
  && V(\x)
     \ar[ru]^{F}\\
}
\]

\item The ``corresponding" cube in between various $\calc$'s commutes.

\item For all $\g \in \G, \la_i \in G_i$, and $x_i \in X_i$ with $\la_i
\in \la_{x_i}\ \forall i$, the ``Verma cube" below, commutes (here,
$Z(\x) := \otimes_i Z_i(x_i)$):
\begin{equation}\label{Ediag3}
\xymatrix{
  & \{ Z_i(\g(\la_i)) \}
     \ar[rr]^{\rtimes}
     \ar[dd]^{\otimes}
  && \{ Z_i(x_i) \}
     \ar[dd]^{\otimes}\\
  \{ Z_i(\la_i) \}
     \ar[rr]^{\rtimes}
     \ar[dd]^{\otimes}
     \ar[ru]^{\g(\cdot)}
  && \{ Z_i(x_i) \}
     \ar[dd]^{\otimes}
     \ar[ru]^{\g(\cdot) = \id}\\
  & Z(\g(\la))
     \ar[rr]^{\rtimes}
  && Z(\x)\\
  Z(\la)
     \ar[rr]^{\rtimes}
     \ar[ru]^{\g(\cdot)}
  && Z(\x) 
     \ar[ru]^{\g(\cdot) = \id}\\
}
\end{equation}
\end{enumerate}
\end{prop}

\noindent We also get a commuting cube by replacing all $Z$'s by $V$'s in
part (4). Note also that the poset structure on each set in part (1) is
needed to prove that the $\calo$'s are highest weight categories.\medskip

We now state our main theorems; the rest of this article is devoted to
proving them. Note that $\calo$ is taken to mean the BGG Category in any
of the four setups. Finally, for an explicit statement of parts 1(b) and
1(c) below, see Proposition \ref{Pequiv}.

\begin{theorem}\label{Tmain}
Suppose the standing assumptions \ref{St3} above hold.
\begin{enumerate}
%%%%%%
%\item If $\G_i$ is trivial for all $i$, then $G_A = \times_i G_{A_i}$
%indexes the simple objects in $\calo_A$. More precisely, $V_A(\la_1,
%\dots, \la_n) := \otimes_i V_{A_i}(\la_i)$ is simple in $\calo$.
%Moreover, $S_A(\la_1, \dots, \la_n) = \times_i S_{A_i}(\la_i)$.
%%%%%%
%
%
\item Each of the following conditions holds in one setup (see
\eqref{E1}), if and only if it holds in any of the other three:
\begin{enumerate}
  \item Finite-dimensional modules in $\calo$ are completely reducible.

  \item For a fixed $\la = (\la_1, \dots, \la_n) \in \times_i G_i$,
  $V(x)$ is finite-dimensional for any simple object $x$ over $\la$.

  \item For a fixed $\la = (\la_1, \dots, \la_n) \in \times_i G_i$,
  $Z(x)$ has finite length for any simple object $x$ over $\la$.

  \item $\calo$ (equivalently, every Verma module $Z(x)$) is finite
  length.

  \item Any of Conditions (S1)-(S4) holds (for a fixed $1 \leq n \leq
  4$); for (S4), we also need that $\mf{Z}(\ag) = \mf{Z}(A)^{\G} \subset
  A$.
\end{enumerate}

\item We have the following sequence of implications of the Conditions
(S): $(S3) \Rightarrow (S2) \Rightarrow (S1)$; moreover, $(S4)
\Rightarrow (S3)$ if $\mf{Z}(\ag) \subset A$.
\end{enumerate}
\end{theorem}

\begin{theorem}\label{Timplies}
Suppose the standing assumptions \ref{St3} above hold.
\begin{enumerate}
\item If $\ag$ satisfies Condition (S1), then $\calo$ is finite length,
and hence splits into a direct sum of abelian, finite length blocks
$\calo(x)$.

\item If $\ag$ satisfies Condition (S2), then each block has enough
projectives, each with a filtration whose subquotients are Verma modules.

\item If $\ag$ satisfies Condition (S3), then each block $\calo(x)$ is
a highest weight category, equivalent to the category $\bfg$ of finitely
generated right modules over a finite-dimensional $k$-algebra $B =
B_x$.
\end{enumerate}
\end{theorem}

\noindent In the last part, moreover, many different notions of block
decomposition all coincide, and (a modified form of) BGG Reciprocity (see
\cite{BGG1}) holds in $\calo$ with a symmetric (modified) Cartan matrix.

\begin{remark}\label{R1}
One can add other (equivalent) setups, for example $A \rtimes \G'$, where
$\G'$ is any finite group that acts ``nicely" on $A,B_\pm,G,\dd$.
\end{remark}\medskip

\noindent {\bf Example:}
Suppose $\G = S_n$ and $A = (\mf{Ug})^{\otimes n}$ for a complex
semisimple Lie algebra $\mf{g}$. Then the ``dominant integral" $x \in X$
are precisely the simple objects over the dominant integral weights
$(P_{\mf{g}}^+)^n$ of $\mfg^{\oplus n}$. Every finite-dimensional module
is in $\calo_{A \rtimes \G}$, and is completely reducible.

Theorem A.$5$ in \cite{Mac} explicitly describes each $x \in X$ (the
construction for $V(x)$ is similar):
choose a partition $n_1 + \dots + n_l = n$, and for each $j$, pick $\la_j
\in \mf{h}^*$ pairwise distinct, and simple $S_{n_j}$-modules $N_j$. Then
$x = \Ind \otimes_{j=1}^l [(\C^{\la_j})^{\otimes n_j} \otimes_{\C} N_j]$,
where one induces from $\otimes_{j=1}^l (S_{n_j} \wr \mf{Uh})$ to $S_n
\wr \mf{Uh}$.

For example, fix any $\la \in \mf{h}^*$. The module corresponding to the
partition $n = n_1$ and the one-dimensional simple $S_n$-module, is $V(x)
= V(\la)^{\otimes n}$ (and $S_n$ acts by permuting the factors). (This
proves the first part of Proposition \ref{Pexample}, since the modules
$V(\la)^{\otimes n}$ have unequal formal characters for pairwise distinct
$\la$.)
More generally, $V(\la)^{\otimes n} \otimes_{\C} E$ is simple in $\calo$
for any simple $S_n$-module $E$.

Finally, all the Conditions (S) hold for $\mf{Ug}$ (and hence for $A
\rtimes S_n$, by above results). This is because by the theory of central
characters and Harish-Chandra's theorem, $S^4(\la) = W_{\mfg} \bullet
\la$ (so the strict HRTA $\mf{Ug}$ satisfies Condition (S4)), where
$W_{\mfg}$ is the (finite) Weyl group of $\mfg$, and $\bullet$ its
twisted action.
This also gives the usual and twisted actions of $S_n \wr W_{\mfg}$ on
the set of weights $(\mf{h}^*_{\mfg})^{\oplus n}$. The twisted action is
``good", since
\[ (\sigma \w) \bullet \la = (\sigma(\w) \sigma) \bullet \la = \sigma( \w
\bullet \la) = \sigma(\w) \bullet \sigma(\la) \]

\noindent for all $\la \in (\mf{h}^*_{\mf{g}})^{\oplus n}, \w \in
W_{\mfg}^n$, and $\sigma \in S_n$. We now adapt the above results here;
note that $R_{\mfg} = S_n \wr \mf{Ug}$.

\begin{prop}\label{Pwreath}\hfill
\begin{enumerate}
\item $R_\mfg$ is of finite representation type.

\item If $\calo'$ is the category of $M \in R_{\mfg}-\Mod$ such that
${\rm Res}^{R_{\mfg}}_A M \in \calo_A$, then $\calo'$ is a direct sum of
subcategories $\calo'(\la)$. For a given summand, the modules in it are
of finite length, and all their simple subquotients have highest weights
only in $S_n ( W_{\mfg}^n \bullet \la )$ for various $\la \in \Theta :=
(\mf{h}^*_{\mfg})^{\oplus n} / (S_n \wr W_{\mfg}, \bullet)$.

\item $\Theta$ is precisely the set of {\em central characters} of
$R_{\mf{g}}$, where $R_{\mf{g}}$ has center $(\mf{Z}(\mf{Ug})^{\otimes
n})^{S_n}$.
\end{enumerate}
\end{prop}

\noindent Proposition \ref{Pdiag} and all the theorems in this section
are proved in \S \ref{Sfinal} below, and Proposition \ref{Pwreath} is
shown in Section \ref{Scentral}.

\section{The first setup - skew group rings}

We start by relating $\calo_A$ and $\calo_{A \rtimes \G}$. The first
thing to do is to identify a set which will characterize the simple
objects in $\calo_{\ag}$.

\begin{definition}\hfill
\begin{enumerate}
\item Denote by $\G_1$ the set of singly generated \G-modules.

\item Let $\calc$ denote the abelian category of finite-dimensional
$H$-semisimple $(H \rtimes \G)$-modules.

\item Let $Y$ denote the set of isomorphism classes of objects in $\calc$
generated by a weight vector $v_\la$ (for some $\la \in G$).

\item Let $X \subset Y$ denote the isomorphism classes of simple objects
in \calc.
\end{enumerate}
\end{definition}

If $m \in M$ is a weight vector in $M \in X$ simple, then
\[ M = (H \rtimes \G)m = k\G (k \cdot m) \]

\noindent so that every $M \in X$ is of the form $M = k\G/I$ for some
ideal $I$ (as $\G$-modules).
Now suppose $M \in Y$; thus, $M$ is of the form $k\G / I$. Note that to
fix an $(H \rtimes \G)$-structure on the module involves fixing the
weight vector $v_\la \in k\G/I$. This is equivalent to choosing which
ideal $J$ of $k\G$ to quotient out by, so that $k\G/J \cong k\G/I \in
\G_1$ (as $\G$-modules). Hence,
\begin{equation}\label{EXYZ}
Y = \{ (\la, M, [J]) : \la \in G, M \in \G_1, k\G/J \cong M \},
\end{equation}

\noindent where $J$ is assumed to be an ideal of $k\G$, that annihilates
$v_\la$. (We thus map $1 \mapsto v_\la$, and kill $h - \la(h) \cdot 1$
for all $h \in H$.) Over here, we only take isomorphism classes of ideals
$[J]$. For example, if $v_\la \in M_\la$ generates $M$, then so does $\g
v_\la \in M_{\g(\la)}$.

\begin{definition}
$\la_y$ (for $y \in Y$) is defined to be this $\la$; or more precisely,
$\la_y$ is the $\G$-orbit $[\la] \in G / \G$. (However, we will abuse
notation and say $\la = \la_y$ - or $\la = \la_x$, if $y = x \in X$ -
instead of $\la \in \la_y$.)
\end{definition}

\begin{remark}\label{R2}
Note that given $M \in \G_1$, not all $\la \in G$ occur in the set of
triples in $Y$ (as the first coordinate). For example, if $M$ is the
trivial representation of $\G$, then $\g-1$ annihilates $M$ for all $\g
\in \G$, so the only permissible weights here are $\la \in G^\G$.

However, for all $\la \in G$, there exists $y \in Y$ such that $\la =
\la_y$. For we take the one-dimensional $H$-module $k^\la = k \cdot
v_\la$, where $h v_\la = \la(h) v_\la$ for all $h \in H$. We now induce
this, to get $M = \Ind^{H \rtimes \G}_H k^\la$. Thus $M \cong k\G$ as
$k$-vector spaces, and $(\la,M,0) \in Y$.
\end{remark}

Now consider the special case $\G = 1$; all objects in $X$ are
one-dimensional, and $X = Y = G = \hhh_{k-alg}(H,k)$. These are the
maximal vectors, from which one induces Verma modules. Moreover, all
objects in \calc are $H$-semisimple. This last part is true in general:

\begin{prop}\label{Phgsplits}
Every object of \calc is completely reducible.
\end{prop}

\begin{proof}
Use Proposition \ref{Pss} with $R = H$, $\calp$ the category of
finite-dimensional $H$-semisimple $H$-modules, and $\cald = \calc$ as
above.
\end{proof}

The following is the other main result of this section.

\begin{theorem}\label{T2}
For each $\la \in G$, there exists one - and if $k$ is algebraically
closed, only finitely many - $x \in X$ with $\la = \la_x$.
\end{theorem}

\begin{proof}
We first show that the set is nonempty. Choose $y \in Y$ such that $\la =
\la_y$ (by Remark \ref{R2}), and look at a composition series $0 \subset
E_1 \subset \dots$ (as above) for $E=M_y$, as $(H \rtimes \G)$-modules.
Then all subquotients are $H$-semisimple, so by character theory,
$(E_{i+1} / E_i)_\la \neq 0$ for some $i$. Then $\la = \la_x$, where $x =
E_{i+1} / E_i \in X$.

If $\G$ is trivial, then so is the second part of the result. Now suppose
that $|\G| > 1$ and $k$ is algebraically closed of characteristic zero.
To show that this set is finite, use \cite[Theorem A.5]{Mac}, setting $R
= H$. Now, the set of subgroups of $\G$ is finite; hence so is the set
$\{ (\G', M)\}$, where $\G'$ is a subgroup of $\G$, and $M$ is a simple
$\G'$-module (up to isomorphism).

The theorem now says that every simple $H \rtimes \G$-module is of the
form $\Ind^{H \rtimes \G}_{H \rtimes \G'} (k^{\la} \otimes (\G')^\mu)$
for some simple $H$-module $k^{\la}$ and simple $\G'$-module $(\G')^\mu$,
where $\G'$ fixes $\la$. (Note that since we are only concerned here with
finite-dimensional representations, a simple $H$-module is
one-dimensional by Lie's theorem; we denote it by $k^{\la}$ as above.)
The structure is given here by
\[ (h \otimes \g) (v_\la \otimes w_\mu) = h \g \cdot v_\la \otimes  \g
\cdot w_\mu = \la(h) \cdot \g v_\la \otimes \g w_\mu \ \forall h \in H,\
\forall \g \in \G' \]

\noindent (where we fix a group action $: \G' \to \Aut_k(k^{\la})$).
Fixing $\la$ and $\g$, the set of $(\G',\mu)$'s is finite, hence we are
done.
\end{proof}

\section{Verma and standard modules}

Unless otherwise specified, the functor $\Ind$ denotes $\Ind^{A \rtimes
\G}_{\hbg}$ henceforth. Given a finite-dimensional $\hbg$-module $E$, we
can define the induced module $\Ind E \in \calo$. Given a
finite-dimensional $(H \rtimes \G)$-module $E$, we also define the {\em
(universal) standard module} $\dd(E)$ as follows: we first give $E$ an
$\hbg$-module structure, by
\[ (h \otimes n_+ \otimes \g)e = 0,\ (h \otimes 1 \otimes \g) e = h \g e
\]

\noindent for each $h \in H,\ \g \in \G,\ e \in E$, and $n_+ \in N_+$.
Now define the induced module $\dd(E)$, to be $\dd(E) = \Ind E$. The
following properties are standard.

\begin{prop}\label{P2}
Suppose $E$ is a finite-dimensional \hbg-module.
\begin{enumerate}
\item $\Ind E  \cong B_- \otimes_k E$, as free $B_-$-modules.

\item If $E$ has a basis of weight vectors, then $\ch_{\Ind E} =
\ch_{B_-} \ch_E$.

\item If $E' \subset E$ is an \hbg-submodule and $E$ is $H$-semisimple
(as above), then $\Ind (E/E') \cong \Ind E / \Ind E'$.
\end{enumerate}
\end{prop}

\noindent We now recall a few concepts from \cite{Kh3}:
\begin{enumerate}
\item $G_0$ has a partial ordering (via the base of simple roots $\dd$):
$\mu \leq \la$ if and only if there exists $\theta_0 \in \nn \dd$ such
that $\theta_0 * \mu = \la$. This induces a partial order on $G$: $\la
\geq \mu$ if $\la = \mu$ in $G$, or $\pi(\la) > \pi(\mu)$ in $G_0$.

\item The {\em formal character} of $M \in \calo$ is $\ch_M := \sum_{\la
\in G} (\dim_k M_\la) e(\la)$.
\end{enumerate}

\noindent (Note that Proposition \ref{P2} used the latter notion.) We now
relate the $\G$-action to the partial ordering.

\begin{stand}\label{St2}
Henceforth, $\G$ acts by order-preserving transformations on $G_0$. In
other words, $\mu \leq \la \Rightarrow \g(\mu) \leq \g(\la)\ \forall \g
\in \G$.
\end{stand}

\begin{remark}
For instance, if $\G = 1$, or $\G = S_n$ and $A \rtimes \G$ is the wreath
product $S_n \wr A$, then this assumption holds. It also clearly holds
when $A \rtimes \G$ is built up from subalgebras $A_i \rtimes \G_i$ (as
discussed above), and each $\G_i$ preserves $\dd_i \subset (G_i)_0$.
Moreover, this assumption is reasonable, as the subsequent lemma shows.
\end{remark}

\begin{lemma}\label{L2.1}
Suppose a set $G_0$ contains a free abelian group (denoted by $\Z \dd_0 =
\bigoplus_{\aaa \in \dd_0} \Z \aaa$) with a free action $*$ on $G_0$,
that restricts to addition on $\Z \dd_0$. Define a partial order on $G_0$
by: $\la \geq \mu$ if and only if $\la \in (\nn \dd_0) * \mu$. Suppose
also that a group $\G_0$ acts on $G_0$, preserving $\Z \dd_0$ and the
action $*$.

\begin{enumerate}
\item The following are equivalent, for a given $\g \in \G_0$:

\begin{enumerate}
\item $\gpm(\aaa) \in \dd_0$ for all $\aaa \in \dd_0$.

\item For each $\aaa \in \dd_0$, there exists $n_{\aaa} \in \N$, so that
$\gpm(n_{\aaa} \aaa) > 0$.

\item If $\aaa \in \dd_0$ then $\g^{\pm 1}(\aaa) > 0$.

\item If $\la \geq 0$ then $\gpm(\la) \geq 0$.

\item $\g,\g^{-1}$ act on $G_0$ by order-preserving automorphisms.
\end{enumerate}

\item If the conditions in the first part are satisfied, and $\g$ has
finite order in $\G_0$, then $\g(\la) \nless \la$ for all $\la \in G_0$.

\item Under the assumptions of Standing Assumption \ref{St1}, if for all
$\aaa \in \dd$, there exists $n_\aaa \in \N$ with $(B_+)_{n_\aaa \aaa}
\neq 0$, then the first part holds for $G_0 = G, \G_0 = \G, \dd_0 = \dd$.
\end{enumerate}
\end{lemma}

\noindent Note that we do not need $\dd_0$ or $\G_0$ to be finite in this
result.

\begin{proof}\hfill
\begin{enumerate}
\item The cyclic chain of implications is easy to prove.

\item If $\g(\la) < \la$ for some $\la,\g$, then $\g^{i+1}(\la) <
\g^i(\la)\ \forall i$, by the previous part. Hence if $\g^n = 1$ in $\G$,
then we get a contradiction:
\[ \la = \g^n(\la) < \g^{n-1}(\la) < \dots < \g(\la) < \la. \]

\item For $\g \in \G, \la \in G,\ \g : A_\la \to A_{\g(\la)}$ by Lemma
\ref{Lcompat} (since $A$ is an $\ad H$-module). In particular, $0 \neq
\g((B_+)_{n_{\aaa} \aaa}) \subset A_{\g(n_{\aaa} \aaa)} \cap B_+$ (by
assumption) $= (B_+)_{\g(n_{\aaa} \aaa)}$. (This makes use of the fact
that $\g \circ \pi = \pi \circ \g$ on $G$.) Hence $\g(n_{\aaa} \aaa) > 0\
\forall \aaa, \g$; now use the first part.
\end{enumerate}
\end{proof}

We now introduce the following notation: for $y \in Y$, we write $y =
(\la_y,M_y,J_y)$ (see equation \eqref{EXYZ}). This representation of $y$
may not be unique, e.g., under the action of \G. We also define {\em
Verma modules} to be $Z(y) = \dd(M_y)$, for $y \in Y$.
The next result is standard; the first part uses Standing Assumption
\ref{St2}, via (the second part of) Lemma \ref{L2.1}.

\begin{prop}\label{P6}\hfill
\begin{enumerate}
\item If $y \in Y$ and $(M_y)_\la \neq 0$, then $(Z(y))_{\la} =
(M_y)_{\la}$, where $\la = \la_y$.

\item If $x \in X$, then $Z(x)$ has a unique simple quotient $V(x)$. Its
``highest weight" vectors also span $M_x$ (as is true in $Z(x)$).

\item $Z(x)$ is indecomposable for $x \in X$.
\end{enumerate}
\end{prop}

We next turn to standard cyclic modules, namely, modules generated by a
single {\em maximal} (i.e., in $\ker N_+$) weight vector.

\begin{definition}\label{D2}\hfill
\begin{enumerate}
\item A {\em standard cyclic module} is a quotient of $\dd(M_y)$ for $y
\in Y$.

\item A module $M$ has an {\em SC-filtration} or a {\em $p$-filtration}
(denoted by $M \in \pfilt$; see \cite{BGG1}) if it has a finite
filtration whose successive quotients are standard cyclic or Verma
modules, respectively.

\item A module $M$ has a {\em simple Verma flag} if it has a
$p$-filtration by Verma modules $\{ Z(x) : x \in X \}$.

\item We define a relation on $Y$: we say $y \leq y'$ if and only if
$\la_y < \g(\la_y')$ in $G$ for some $\g \in \G$, or else $y = y'$.
\end{enumerate}
\end{definition}

\begin{prop}\label{P5}\hfill
\begin{enumerate}
\item $\leq$ is a partial order on $Y$.

\item If $E$ is any finite-dimensional $H$-semisimple \hbg-module, then
$\Ind E$ has a simple Verma flag.
\end{enumerate}
\end{prop}

\begin{proof}\hfill
\begin{enumerate}
\item If $y \leq y'$ and $y' \leq y$, then $y=y'$ by Lemma \ref{L2.1}.

\item Look at a composition series for $E$ in \calc, say $0 \subset E_1
\subset \dots \subset E_n = E$, with $E_{i+1} / E_i \cong M_{x_i}$ for
some $x_i \in X$. Using formal characters, we can rearrange the $E_i$'s
(see \cite{BGG1}), such that $x_i \geq x_j \Rightarrow i \leq j$.

Character theory now shows that this ``rearranged" filtration is a chain
of $\hbg$-modules. But then $0 \subset \Ind E_1 \subset \dots \subset
\Ind E_n = \Ind E$ is a simple Verma flag for $\Ind E$, by the exactness
of $\Ind$ (from Proposition \ref{P2} above).
\end{enumerate}
\end{proof}

\section{Simple modules}

We now classify all simple modules in $\calo$, as well as those of them
which are finite-dimensional. We assume that we have classified all
finite-dimensional simple $A$-modules $V_A(\la) \in \calo_A$. (Note, as
in \cite{Kh}, that if $k$ is algebraically closed, then all
finite-dimensional simple $A$-modules are in $\calo_A$, and hence are of
the form $V_A(\la)$ for some $\la$.) The following is trivial.

\begin{lemma}\label{Lmax}
Given $M \in \calo$, a weight vector $v \in M$ is maximal, if and only if
so is $\g v$ for any $\g \in \G$.
\end{lemma}

\begin{prop}[``Weyl Character Formula 1"]\label{P7}
Fix $x \in X$, and consider $M_x \subset V(x)$. Then $B_- v_\mu =
V_A(\mu)$ for all weight vectors $v_\mu \in M_x$, and left-multiplication
by $\g$ is a vector space isomorphism : $V_A(\mu) \to V_A(\g(\mu))$.
Thus, $\ch_{V_A(\g(\mu))} = \g(\ch_{V_A(\mu)})$ for all $\g, \mu$.

Moreover, if $\{ \g_i v_{\la_x} : 1 \leq i \leq n \}$ is a weight basis
of $M_x$, then $V(x) = \bigoplus_i V_A(\g_i(\la_x))$ as $A$-modules. In
particular, $\ch_{V(x)} = \sum_{i=1}^n \ch_{V_A(\g_i(\la_x))}$.
\end{prop}

\begin{proof}
We first note that if $B_- v_\mu = V_A(\mu)$ for all $v_\mu \in M_x$,
then $\g B_- v_\mu = \g B_- \g^{-1} \cdot \g v_\mu = B_- \g v_\mu$, and
this must equal $V_A(\g(\mu))$ (it is simple because any maximal vector
in $\g V_A(\mu)$ must come from one in $V_A(\mu)$). Thus the second part
follows from the first (and holds for every $\mu$, since each $\mu$ is of
the form $\la_x$ for some $x \in X$, from above results).

Observe that $v_\mu$ is maximal, being in $M_x$, hence $B_- v_\mu$ is a
standard cyclic module in $\calo_A$. We claim that it is simple. Suppose
not. Then there exists a maximal vector $b_- v_\mu \in B_- v_\mu$, of
weight $\nu <\mu$, say.
We now claim that $V := \sum_\g B_- \g \cdot (b_- v_\mu) = \sum_\g \g B_-
(b_- v_\mu)$ is a (nonzero) proper $(A \rtimes \G)$-submodule of $V(x)$
(which is a contradiction). For since the $\g b_- v_\mu$'s have weights
$\g(\nu) < \g(\mu)$ respectively, are maximal vectors by Lemma
\ref{Lmax}, and generate $V$, hence $v_\mu \notin V$ by Lemma \ref{L2.1}.
Thus $B_- v_\mu = V_A(\mu)$, as claimed.

Finally, if $\{ \g_i v_{\la_x} \}$ is a basis of $M_x$, then
\[ V(x) = B_- M_x = B_- (k\G \cdot v_{\la_x}) = \sum_i B_- \g_i v_{\la_x}
= \sum_i V_A(\g_i(\la_x)) \]

\noindent Now, $\g_i(v_{\la_x}) \notin V' := \sum_{j \neq i}
V_A(\g_i(\la_x))$, so $V_A(\g_i(\la_x)) \cap V' = 0$. Hence the above sum
of simple $A$-modules is direct; now use character theory.
\end{proof}

\begin{remark}\label{R4}
The same results hold if we replace simple modules by Verma modules,
i.e., $x \in X$ by $y \in Y,\ V$'s by $Z$'s, and $V_A$'s by $Z_A$'s
respectively (using the proof of Proposition \ref{P2}). Moreover, we
claim that the length $l_A(Z_A(\g(\la)))$ (if finite) is independent of
$\g$, for any $\la \in G$. For, choose any $x \in X$ such that $\la_x =
\la$; then $B_- \g v_\la \cong Z_A(\g(\la))$ as $A$-modules, for all
$\g$. Moreover, $\g : Z_A(\la) \to Z_A(\g(\la))$ takes maximal vectors to
maximal vectors, so it preserves the length of any filtration.
\end{remark}

\begin{cor}\label{C1}
Suppose $V_A(\la)$ is finite-dimensional. Then so is $V_A(\g(\la))$ for
all $\g$, and the dimension is independent of $\g$. Moreover, $\dim V(x)
= \dim M_x \cdot \dim V_A(\la)$, if $\la = \la_x$.
\end{cor}

\begin{proof}
For all $\mu \in G,\ \g : V_A(\mu) \to V_A(\g(\mu))$ is a vector space
isomorphism (as $A$-submodules of some $V(x)$), by Theorem \ref{T2} and
Proposition \ref{P7}. The next part is also clear by Proposition
\ref{P7}.
\end{proof}

All these facts come together in proving

\begin{theorem}\label{Tfacts}\hfill
\begin{enumerate}
\item Every simple module in \calo is of the form $V(x)$ for some $x \in
X$.

\item Given $x \in X$, the simple module $V(x)$ is finite-dimensional if
and only if $V_A(\la_x)$ is finite-dimensional.

\item Given $x \in X$, the Verma module $Z(x)$ is of finite length, if
and only if $Z_A(\la_x) \in \calo_A$ is.
\end{enumerate}
\end{theorem}

\begin{proof}\hfill
\begin{enumerate}
\item This is standard from the previous section, say.

\item If $V(x)$ is finite-dimensional, then so is its $A_i$-submodule
$V_A(\la_x)$ (by Proposition \ref{P7}). The converse follows from
Corollary \ref{C1}.

\item If each $Z_A(\la_x) \in \calo_A$ has finite length, then by Remark
\ref{R4}, $Z(x)$ has finite length as an $A$-module - hence also as an $A
\rtimes \G$-module. Conversely, if $l_{A \rtimes \G}(Z(x)) < \infty$,
then $l_A(Z(x)) < \infty$, since by Proposition \ref{P7}, $l_A(V(x')) <
\infty\ \forall x'$. Now use Remark \ref{R4} again.
\end{enumerate}
\end{proof}

\begin{cor}
If $k$ is algebraically closed, then every finite-dimensional simple $A
\rtimes \G$-module $V$ is of the form $V(x) \in \calo$ for some $x \in
X$.
\end{cor}

\section{Duality}\label{Sduality}

We now introduce the duality functor, that helps obtain information about
the Ext-quiver in $\calo$ (and its relation to the partial order on the
simple objects). We first make a general definition. Suppose we have a
$k$-algebra $A'$, satisfying the following:\medskip

\noindent There exists an anti-involution $i : A' \to A'$, that fixes a
subalgebra $H' \subset A'$.

\begin{definition}\hfill
\begin{enumerate}
\item The {\em Harish-Chandra category} $\calh' = \calh_{A',H'}$ over
$(A',H')$ consists of all $A'$-modules $M$ with a simultaneous weight
space decomposition for $H'$, and finite-dimensional weight spaces.

\item The {\em duality functor} $F : \calh' \to \calh'$ is defined as
follows: $F(M)$ is the span of all $H'$-weight vectors in $M^* =
\hhh_k(M,k)$. It is a module under: $\tangle{a' m^*, m} = \tangle{m^*,
i(a') m}$ for $a' \in A', m \in M, m^* \in F(M)$.
\end{enumerate}
\end{definition}

\begin{remark}
One can then show \cite[Propositions 1,2]{Kh} (except for part 2 of
Proposition (2.2)):

\begin{prop}
$F$ is exact, contravariant, and preserves lengths and formal characters.
Moreover, $F(F(M)) = M$ for all $M \in \calh'$.
\end{prop}
\end{remark}

\subsection{Functoriality}
Now suppose we have algebras $A' \supset A'' \supset H'$, with $i_{A'}
|_{A''} = i_{A''}$. We then have the duality functors $F', F''$ on the
Harish-Chandra categories $\calh' = \calh_{A', H'}, \calh'' = \calh_{A'',
H'}$ respectively, and the forgetful functor $\id' : \calh' \to \calh''$.
The following is easy to prove.

\begin{lemma}\label{Lduality}
$F'' \circ \id' = \id' \circ F'$ on $\calh'$.
\end{lemma}

\subsection{Application to the BGG Category}
Note that $A \rtimes \G$ has an anti-involution $i = i_A \otimes i_\G$ by
the standing assumptions. This enables us to define the duality functor
$F : \calo \to \calo^{op} \subset \calh$, as in \cite{Kh}. Now, $F$
permutes the set of simple objects, so $F(V(x))$ is also a simple object
in $\calh = \calh_{A \rtimes \G, H}$. Moreover, $\G$ acts on formal
characters (i.e., on $\Z[G]$): if $e(\la) \in \Z[G]$ corresponds to $\la
\in G$, then $\g(e(\la)) = e(\g(\la))$.

We now put $A' = H \rtimes \G$ and $H' = H$. Then the analogous results
hold, and we get a duality functor on $\calh_{H \rtimes \G}$, that
restricts to one on \calc as well. In particular, $F$ permutes the set of
simple objects, i.e., $F : X \to X$. For example, if $\G = 1$, then
$F(\la) = \la \in G$, since each $x \in X = G$ is then one-dimensional.

The following result relates the dualities in $\calc$ and $\calo$, and
generalizes part 2 of \cite[Proposition 2.2]{Kh}.

\begin{prop}\label{P4}
For all $x \in X$, we have $F(V(x)) = V(F(x))$.
\end{prop}

\begin{proof}
We know that $F(V(x))$ is a simple module in $\calh$ with the same formal
character as $V(x)$. It thus has a weight vector of maximal weight, which
generates the entire module, since it is simple. Thus $F(V(x)) \in
\calo$, whence it is of the form $V(x')$ from above. We claim that $x' =
F(x)$.

To see this, apply Lemma \ref{Lduality}, setting $A' = A \rtimes \G, A''
= H \rtimes \G, H' = H$. The ``highest" set of weight spaces in $V(x)$ is
the $H \rtimes \G$-module $M_x$, and $F''$ sends it to $M_{F(x)}$; now by
Lemma \ref{Lduality}, $F(V(x)) = V(F(x))$.
\end{proof}

We conclude with a standard variant of Schur's Lemma, in addition to
dualizing the fact that $V(x)$ is the unique simple quotient of $Z(x)$.

\begin{prop}\label{P10}\hfill
\begin{enumerate}
\item For all $x \in X$, $F(Z(x))$ has socle $V(F(x))$.

\item Given $x,x' \in X$,
\begin{eqnarray*}
&&    \Hom(Z(x), F(Z(F(x'))))     = \Hom(Z(x), V(x'))\\
& = & \hhh_{\hbg}(Z(x), M_{x'})   = \hhh_{\hbg}(V(x), M_{x'})\\
& = & \hhh_{\hbg}(M_x, M_{x'})    = \hhh_{\hbg}(M_x, V(x'))\\
& = & \Hom(V(x), V(x'))           = \delta_{x,x'} \End_{H \rtimes
\G}(M_x),
\end{eqnarray*}

\noindent where each $M_x$ is killed by $N_+$, and $M \twoheadrightarrow
M_x$ for $M = Z(x), V(x)$, with kernel(s) $\oplus_{\mu < \g(\la_x)}
M_\mu$.

\item $\dim_k \End(M_x) = \dim_k \End(M_{F(x)})\ \forall x \in X$.

\item If $k$ is algebraically closed when $|\G|>1$, then $M_x, V(x)$, and
$Z(x)$ are {\em Schurian}, i.e., the only module endomorphisms are
scalars.
\end{enumerate}
\end{prop}

\noindent We merely remark that in the second part, the first equality is
standard. Moreover, for all terms but the first and the last, all
$\hhh$-groups are clearly zero unless $x = x'$, in which case, we produce
a cyclic chain of maps $\varphi_{AB} : A(x) \to B(x')$, that compose to
give the identity:
\[ \varphi_{ZV} \mapsto \varphi_{ZM} \mapsto \varphi_{VM} \mapsto
   \varphi_{MM} \mapsto \varphi_{MV} \mapsto \varphi_{VV} \mapsto
   \varphi_{ZV}. \]

\section{Homological properties}

In this section, we show some results that are needed in later sections.

\subsection{Every object has a good filtration}

We now show that every module in $\calo$ has a filtration with standard
cyclic subquotients. As in \cite{Kh}, we need more notation.

\begin{definition}\hfill
\begin{enumerate}
\item Define $\hgt : \Z \dd \to \Z$, via: $\hgt \left( \sum_{\aaa \in
\dd} n_\aaa \aaa \right) := \sum_\aaa n_\aaa$.

\item Given $l \in \nn$, define $B_{+l} := \sum_{\hgt \pi(\theta) \geq l}
(B_+)_\theta$ in $B_+$ for all $l \in \N$.

\item Given $\la \in G$ and $l \in \N$, define the subcategory
$\calo(\la, l)$ to be the full subcategory of all $M \in \calo$ such that
$B_{+l} \cdot k\G \cdot M_\la = 0$.

\item Given $y \in Y$, define the module
\[ P(y,l) := (A \rtimes \G) / (B_{+l}, J_y), \]

\noindent where $y = (\la_y,M_y,[J_y])$, and we identify 1 with the
generator $v_{\la_y}$ (by choice of $[J_y]$), as in equation
\eqref{EXYZ}. (Thus $h - \la_y(h) \cdot 1 \in J_y$ for all $h \in H$.)

\item Define $I_{y,l}$ to be the left ideal of (the $k$-algebra) \hbg
generated by $B_{+l}$ and $J_y$, and set $E_{y,l} := (\hbg) / I_{y,l}$.
\end{enumerate}
\end{definition}\medskip

\noindent Thus, $\Ind E_{y,l} = B_- \otimes E_{y,l} = (A \rtimes \G) /
((A \rtimes \G) I_{y,l}) = P(y,l)$; in particular, $P(y,l) \in \calo$. In
fact, by considering the formal character of $E_{y,l},\ P(y,l) \in
\calo(\la_y, m)\ \forall m \geq l, y \in Y$; moreover, $P(y,l+1)
\twoheadrightarrow P(y,l)\ \forall y,l$.\medskip

Note that $k\G \cdot \bar{1} = M_y \subset E_{y,l}$. Since $E_{y,l} = B_+
M_y$ is a finite-dimensional $H$-semisimple $\hbg$-module, hence this is
similar to a construction in \cite{BGG1}, and by the above results,
$P(y,l)$ has a simple Verma flag, with a ``suitable arrangement" of
composition factors, by the proof of Proposition \ref{P5} above. In fact,
for all $l$, one of the terms in the Verma flag for $P(x,l)$ is $Z(x)$,
and by inspecting the formal character of $E_{x,l}$ (or its direct sum
decomposition as a $H \rtimes \G$-module), we find that all other terms
are of the form $Z(x')$ for $x'>x$.
For example, if $l=1$, then $B_{+1} = N_+$, so $P(y,1) = Z(y)$.\medskip

We now analyze singly-generated modules in $\calo$, and then all modules.
If $N = (A \rtimes \G) v_\la$ for some (not necessarily maximal) weight
vector $v_\la$, then $N = B_- B_+ (k\G v_\la)$, where we note that $k\G
v_\la = M_y$, say, for some $y \in Y$. Since $N \in \calo$, hence
$B_+(k\G v_\la)$ is finite-dimensional, so that $P(y,l)
\twoheadrightarrow N = (A \rtimes \G) v_{\la_y}$ for all $l \gg 0$. In
fact, we have:

\begin{prop}\label{P8}\hfill
\begin{enumerate}
\item For all $l \geq 0$, $B_{+l}$ is a two-sided ideal of $B_+$ with
finite codimension, that is stable under the $\G$-action.

\item If $N \in \calo(\la,l)$ and $\la = \la_y$, then\\
$\hhom(P(y,l),N) = \hhh_{H \rtimes \G}(M_y, k\G \cdot N_{\la_y})$.

\item $P(y,l)$ is projective in $\calo(\la_y,l)$ for all $y \in Y,\ l \in
\N$.

\item If $N \in \calh$, then the following are equivalent:
\begin{enumerate}
\item $N \in \calo$.

\item $N$ is a quotient of a (finite) direct sum of $P(y_i,l_i)$'s.

\item $N$ has an SC-filtration, with subquotients of the form $Z(x) \to V
\to 0$, with $x \in X$.
\end{enumerate}
\end{enumerate}
\end{prop}

\begin{proof}\hfill
\begin{enumerate}
\item Two-sidedness follows from (RTA7), and finite codimention from
(RTA6). Next, if $b \in B_{+l}$ is a weight vector of weight $\la$, then
$\hgt \la \geq l$. But then by Lemma \ref{L2.1}, $\hgt \g(\la) \geq l$,
whence $\g(b) \in B_{+l}$.

\item This is as in \cite{BGG1} (or also \cite[Proposition 5]{Kh}).

\item This follows from the previous part and Proposition \ref{Phgsplits}
above.

\item This is proved as in \cite[Proposition 7]{Kh}.
\end{enumerate}
\end{proof}

A standard consequence is

\begin{prop}\label{Pfinlength}
The following are equivalent:
\begin{enumerate}
\item $Z_A(\la)$ has finite length for all $\la$.
\item $Z(x)$ has finite length (as an $\ag$-module) for all $x \in X$.
\item $Z(y)$ has finite length for all $y \in Y$.
\item $\calo_A$ is finite length.
\item $\calo_{\ag}$ is finite length.
\end{enumerate}
\end{prop}

\subsection{Extensions between simple objects}

We now obtain information about the Ext-quiver in $\calo$. Recall the
partial ordering on $Y$, namely: $y \leq y'$ if and only if $\la_y <
\g(\la_{y'})$ for some $\g \in \G$, or $y=y'$. We also define $Y(x)$ (for
any $x$) to be the unique maximal submodule of $Z(x)$.
We now imitate a result in \cite{Kh3}; the proof is similar to that of
\cite[Proposition 4, part 2]{Kh}.

\begin{prop}\label{Pext}
$E_{x,x'} := \Ext^1(V(x), V(x'))$ is nonzero if and only if ($x > x'$
and) $Y(x) \twoheadrightarrow V(x')$, or ($x' > x$ and) $Y(F(x'))
\twoheadrightarrow V(F(x))$.

\noindent Moreover, $E_{x,x'} \cong E_{F(x'),F(x)}$ is
finite-dimensional.
\end{prop}

We now show our first block decomposition - though not for $\calo$. The
proof uses Proposition \ref{Pext} and an argument similar to
\cite[Theorem 4]{Kh}; each ``finite length block" is clearly self-dual.

\begin{prop}\label{Pskewblock}
Define $\caloo$ to be the full subcategory of all finite length objects
in $\calo$. The sum $\sum_{x \in X} \left( \calo(x) \cap \caloo \right)$
of distinct subcategories (of all finite length objects in $\calo(x)$) is
direct, and equals all of $\caloo$. Each ``finite length block" is
abelian, (finite length,) self-dual, and closed under morphisms and
extensions (in $\caloo$). All morphisms and extensions between distinct
blocks are trivial.
\end{prop}

\section{The Conditions (S)}

\begin{prop}\label{P9}
Fix $y \in Y, \mu \in G$. The following are equivalent:
\begin{enumerate}
\item $[Z_A(\g(\la_y)) : V(\mu)] > 0$ for some $\g \in \G$.

\item There exists $x \in X$ so that $\la_x = \mu$ and $[Z(y) : V(x)] >
0$.
\end{enumerate}
\end{prop}

For this result (and also later), we need the following lemma.

\begin{lemma}\label{Lsummand}
For any ring $R$, every simple (sub)quotient of a direct sum of
$R$-modules, is automatically a simple (sub)quotient of some summand.
\end{lemma}

\begin{proof}[Proof of the proposition]
First assume that (1) holds; then there is a $b_- \in B_-$ with $b_-
w_{\g(\la_y)} = w_\mu$, where $w_{\g(\la_y)}, w_\mu$ are maximal weight
vectors of appropriate weights. Now assume $v_{\la_y}$ is maximal in
$Z(y)$. Then we claim that $v_\mu := b_- \g v_{\la_y}$ is maximal in
$Z(y)$; this is because $B_- \cdot \g v_{\la_y} \cong Z_A(\g(\la_y))$ as
$A$-modules.

Let us now write $k\G \cdot v_\mu$ as a direct sum of simple
$\hbg$-modules, by results above. If $M_x$ is any simple summand, then we
see that $V(x)$ is a subquotient of $Z(y)$, with $\la_x = \mu$; thus, (2)
follows.

Conversely, assume $(2)$. Given a weight basis $\mathcal{B}$ of $M_y
\subset Z(y)$, we have $Z(y) \cong \bigoplus_{v \in \mathcal{B}} B_- v$
as $A$-modules (from Remark \ref{R4}), with each summand a Verma module
for $A$. So if $[Z(y) : V(x)] > 0$, then as $A$-modules, $V_A(\la_x)$ is
a simple subquotient of $Z(y)$. By Lemma \ref{Lsummand}, $V_A(\la_x)$ is
a simple subquotient of some $B_- v$, hence of $Z_A(\g(\la_y))$ for some
$\g$.
\end{proof}

Next, recall the definition of $S^n(\la) = S^n_A(\la), S^n(x)$, for $\la
\in G$ and $x \in x$ (see Definition \ref{D1}). We now relate the
Conditions (S) for $A$ and $\ag$.

\begin{prop}\label{Pfunct}
Define $\sgx := \{ x' \in X : \la_{x'} = \g(\mu)$ for some $\g \in \G,
\mu \in S^3_A(\la_x) \}$. Now given all $\la \in G, \g \in \G, x \in X$,
we have:
\begin{enumerate}
\item $\g(S^n_A(\la)) = S^n_A(\g(\la))$ for $n=1,2,3$.

\item The map $\wt: X \to G/\G$, sending $x \mapsto \la_x$, satisfies:
\begin{equation}\label{Econds3}
\bigcup_{\wt(x) = \la} S^3(x) = \sgx = \wt^{-1}(\G(S^3_A(\la)))\ \forall
x \in \wt^{-1}(\la).
\end{equation}

\item For $n=1,2$, and any $\la \in G$,
\begin{equation}\label{Econds12}
\bigcup_{\wt(x) = \la} S^n(x) = \G(S^n_A({\la})).
\end{equation}
\end{enumerate}
\end{prop}

\noindent Note that in both equations, only the left-hand side
(supposedly) depends on the specific $\la$ inside an ``$S^n$-set".

\begin{proof}\hfill
\begin{enumerate}
\item We first show that $\g$ takes ``edges" in $S^3_A(\la)$ to ``edges"
in $S^3_A(\g(\la))$. Take any $y \in Y$ such that $\la = \la_y$ (such a
$y$ exists, by Remark \ref{R2}). Now consider the Verma module $Z(y)$.
For all $v_{\g(\la_y)} = \g v_{\la_y} \in M_y$, we have the $A$-Verma
module $Z_A(\g(\la_y)) = B_- v_{\g(\la_y)}$. Now use Remark \ref{R4} to
observe that $V_A(\mu)$ is a subquotient of $Z_A(\la) \subset Z(y)$ if
and only if $V_A(\g(\mu))$ is a subquotient of $Z_A(\g(\la))$. (We note
that $v_\mu \in Z_A(\la)$ is maximal if and only if so is $\g v_\mu$, by
Lemma \ref{Lmax}.)

The proof is similar if $[Z_A(\mu) : V_A(\la)] > 0$: start with $y \in Y$
so that $\mu = \la_y$.
Applying transitivity, we conclude that  $\g(S^3_A(\la)) \subset
S^3_A(\g(\la))$. Replacing $\g$ by $\g^{-1}$ and $\la = \la_y$ by
$\g(\la)$, we get:
\[ S^3_A(\la) = \g^{-1}(\g(S^3_A(\la))) \subset \g^{-1}(S^3_A(\g(\la)))
\subset S^3_A(\g^{-1}(\g(\la))) = S^3_A(\la). \]

This proves the result for (S3). The other two subparts now follow, by
using this subpart and Lemma \ref{Lcompat}:
\begin{eqnarray*}
\g(S^2_A(\la)) & = & \g(\pi(S^3_A(\la))) = \pi(\g(S^3_A(\la))) =
\pi(S^3_A(\g(\la))) = S^2_A(\g(\la)),\\
\g(S^1_A(\la)) & = & \g(S^2_A(\la)^{\leq \la}) = S^2_A(\g(\la))^{\leq
\g(\la)} = S^1_A(\g(\la)).
\end{eqnarray*}

\item We show this in two parts.
\begin{enumerate}
\item It is clear that $\wt^{-1}(S^3_A(\la)) = \sgx$ if $\la_x = \la$.
Next, if $\la_x \in \G(\la)$, then we claim that $S^3(x) \subset
\wt^{-1}(\G(S^3_A(\la)))$:

First, $\la_{F(x)} = \la_x \in S^3_A(\la_x)\ \forall x$. Next, if $[Z(x)
: V(x')] > 0$, then using the structure of $V(x'), Z(x)$, we see that
$V_A(\la_{x'})$ is a subquotient of a direct sum of some
$Z_A(\g(\la_x))$'s. By the previous part, Lemma \ref{Lsummand}, and
Proposition \ref{P9}, $\la_{x'} \in S^3_A(\g(\la_x)) = \g(S^3_A(\la_x))$.
Hence $\la_{x'} \in \G(S^3_A(\la_x))$, and by symmetry and transitivity,
the claim is proved.

\item We now show that $\wt^{-1}(\G(S^3_A(\la))) \subset \bigcup_{\wt(x)
= \la} S^3(x)$. Recall the graph structure on $G$, under which
$S^3_A(\la)$ is a connected component. We now prove this inclusion by
induction on $d(-,\la)$, where $d(-,-)$ is the {\em graph distance
function} (and $\la$ is the distinguished weight in the statement).

In what follows, $\la_k$ will denote some element of $S^3_A(\la)$ such
that $d(\la_k,\la) = k \geq 0$. (This is well-defined, since $\g :
S^3_A(\la) \to S^3_A(\g(\la))$ ``takes edges to edges" by the previous
part.)

The result is clear for $k=0$, since $\la_0 = \la$. Now suppose that it
holds for all $\la_k$ (for some fixed $k \geq 0$). Given $\la_{k+1} \in
S^3_A(\la)$, we know there exists a $\la_k$ connected to it by an edge.
Suppose $[Z(\la_{k+1}) : V(\la_k)] > 0$ (the proof of the other case is
similar). Choose any $x_{k+1} \in \wt^{-1}(\la_{k+1})$; then by
Proposition \ref{P9}, there exists $x_k \in \wt^{-1}(\la_k) \subset
\bigcup_{\wt(x) = \la} S^3(x)$ (by the induction hypothesis), such that
$[Z(x_{k+1}) : V(x_k)] > 0$. But then $x_{k+1} \in \bigcup_{\wt(x) = \la}
S^3(x)$ as well, and we are done by induction.
\end{enumerate}

\item By equation \eqref{Econds3}, $\bigcup_{\wt(x) = \la} \wt(S^3(x)) =
\G(S^3_A(\la))$. To prove equation \eqref{Econds12} for $n=2$, apply
$\pi$ to both sides and use Lemma \ref{Lcompat}. To prove the equation
for $n=1$, first intersect both sides with $G^{\leq \la}$, and then apply
$\pi$. We are then done, if we note that
\[ \bigcup_{\wt(x) = \la} S^3(x)^{\leq \la} = \bigcup_{\wt(x) = \la}
S^3(x)^{\leq x}. \]
\end{enumerate}
\end{proof}

We now collect our results relating the Conditions (S) for $A$ and $\ag$.

\begin{theorem}\label{Tconditions}
Suppose $\ag$ is a skew group ring over an RTA.
\begin{enumerate}
\item The various Conditions (S) satisfy: $(S3) \Rightarrow (S2)
\Rightarrow (S1)$.

\item Suppose $k$ is algebraically closed whenever $|\G|>1$. Then $\ag$
satisfies each of the Conditions (S1),(S2),(S3), if and only if $A$ does.

\item If Condition (S1) holds, then $\calo$ is finite length.

\item If $\calo$ is finite length, then it splits into a direct sum of
blocks.
\end{enumerate}
\end{theorem}

\noindent Note that semisimple Lie algebras and quantum groups satisfy
Condition (S4) (as seen earlier), hence all the Conditions (S) as well -
and hence if $k$ is algebraically closed of characteristic zero, then
their wreath products also satisfy all the Conditions (S). (We relate
Condition (S4) across the various setups, in Section \ref{Scentral}.)

\begin{proof}
The first part is by definition, and the last part follows from
Proposition \ref{Pskewblock}.
The second part follows from equations \eqref{Econds3} and \eqref{Econds12},
and Theorem \ref{T2}, because (if $k$ is algebraically closed,) $\wt$ is
a finite-to-one map.

We now abuse notation (using the freeness of the action $*$) to write
$\la - \la' = \theta$ when $\theta * \la' = \la$, and use that $\pi$
intertwines the $*$-actions on $G$ and $G_0$. Now given $x \in X$, a
submodule (or a maximal vector) in $Z(x)$ of restricted highest weight
$\la_0 (\neq \pi(\g(\la_x))) \in G_0$ occurs only if $\la_0 =
\pi(\la_{x'})$ for some $x' < x$ in $S^3(x)$. This yields:
\[ l_{\ag}(Z(x)) \leq \sum_{\la \in S^1(x)} p(\la_0 - \pi(\la_x)), \]

\noindent where $p$ is the {\em Kostant partition function} $: G \to
\nn$, defined by
$p(\theta) := \dim_k(B_-)_\theta$.
But this summation is a finite sum by assumption, and each summand is
finite by regularity of $A$; now use Proposition \ref{Pfinlength}.
\end{proof}

\section{Central characters}\label{Scentral}

Next, we discuss the notion of {\em central characters}, as in
\cite{BGG1} - but over skew group rings now. All but the last subsection
are general; the last focusses on $R_{\mf{g}} = S_n \wr \mf{Ug}$.

\subsection{Central characters for skew group rings}

We start with a few definitions. Given a skew group ring $A \rtimes \G$
over an RTA $A$ (satisfying Standing Assumptions \ref{St1}, \ref{St2}),
recall the triangular decomposition $A \cong B_- \otimes H \otimes B_+$,
the augmentation ideals $N_\pm$ of $B_\pm$, and the (possibly infinite)
sets $S^3_A(\la)$ for each $\la \in G$, defined above.

\begin{definition}
To start with, denote the center of $A \rtimes \G$ by $Z$.
\begin{enumerate}
\item The {\em Harish-Chandra projection} is $\xi := \epsilon^- \otimes
\id \otimes \epsilon^+ : A = B_- \otimes H \otimes B_+ \twoheadrightarrow
H$ (see Proposition \ref{Pbabe}).

\item Given $\la \in G$, the subgroups $\G^\la,\gsl$ of $\G$ are defined
to be:
$\G^\la := \{ \g \in \G : \g(\la) = \la \}, \qquad \gsl := \bigcap_{\mu
\in S^3_A(\la)} \G^\mu$.

\item Given $\la \in G$, the {\em central character} $\chi_\la$ is the
map $: A \rtimes \G \to k\G$, defined as follows: given $r = \sum_{\g \in
\G} a_\g \g$ with $a_\g \in A$, define $\chi_\la(r) := \sum_{\g \in
\G^\la} \la(\xi(a_\g)) \g$.
\end{enumerate}
\end{definition}

This definition is motivated by the fact that in \cite{BGG1} (where $|\G|
= 1$), the center acts via such characters on (maximal vectors in)
objects in $\calo$.

\begin{lemma}\label{Lcent}
If $v_\la$ is a maximal vector (i.e., $N_+ v_\la = 0$) of weight $\la$ in
any $A \rtimes \G$-module $M$, then $z v_\la = \chi_\la(z) v_\la$ in $M$,
for all $z = \sum_\g z_\g \g$ in $Z$.
\end{lemma}

\begin{proof}
Since each $\g v_\la$ is maximal, hence writing $z_\g = n_\g \oplus b_\g
\oplus \xi(z_\g)$ (with $n_\g \in A N_+, b_\g \in N_- H$), we get that
$n_\g$ kills $\g v_\la$, and $b_\g \g v_\la$ has no $\la$-weight
component, by Lemma \ref{L2.1}. Hence:
\[ z v_\la = \sum_{\g \in \G} \xi(z_\g) \g v_\la = \sum_{\g \in \G}
\g(\la)(\xi(z_\g)) \g v_\la. \]

\noindent This is because if $v_\la \in M_\la$ for any $A \rtimes
\G$-module $M$, then we claim that $z v_\la \in M_\la$ as well: for any
$h \in H$, we have $h \cdot z v_\la = z \cdot h v_\la = \la(h) z v_\la$.
Thus, the above summation only runs over $\g \in \G^\la$, so we have
\begin{equation}\label{E4}
z v_\la = \sum_{\g \in \G^\la} \la(\xi(z_\g)) \g v_\la = \chi_\la(z)
v_\la.
\end{equation}
\end{proof}

We now explore properties of central characters. The first result is that
they are compatible with the $\G$-action, in the following sense:

\begin{prop}\label{Pcent}
Given $r \in A \rtimes \G, \la \in G, \beta \in \G$, we have
\[ \chi_{\beta(\la)}(r) = \beta \chi_\la(\beta^{-1}(r)) \beta^{-1}. \]
\end{prop}

\begin{proof}
Suppose $r = \sum_{\g \in \G} a_\g \g$, with $a_\g \in A\ \forall \g$.
Then
\[ \chi_{\beta(\la)}(r) = \sum_{\g \in \G^{\beta(\la)}}
\tangle{\beta(\la), \xi(a_\g)} \g = \sum_{\g \in \G^{\beta(\la)}}
\tangle{\la, \xi(\beta^{-1}(a_\g))} \g, \]

\noindent where the second equality holds because $\g(\xi(a)) =
\xi(\g(a))$ for all $a \in A, \g \in \G$ (since $\Ad \g$ preserves
$N_\pm$). Therefore
\begin{align*}
\chi_\la(\beta^{-1}(r)) = & \chi_\la \left( \sum_{\g \in \G}
\beta^{-1}(a_\g) \cdot \beta^{-1} \g \beta \right) = \sum_{\beta^{-1} \g
\beta \in \G^\la} \la(\xi(\beta^{-1}(a_\g))) \beta^{-1} \g \beta\\
= & \beta^{-1} \cdot \sum_{\g \in \G^{\beta(\la)}}
\la(\xi(\beta^{-1}(a_\g))) \g \cdot \beta = \beta^{-1}
\chi_{\beta(\la)}(r) \beta
\end{align*}

\noindent using the above equation.
\end{proof}

Next, we show that at least on the center, central characters have very
few nonzero components:

\begin{prop}\label{Pcenter}
Fix $\la \in G$. If $\g \notin \gsl$, and $z = \sum_\g z_\g \g$ is any
central element, then $\la(\xi(z_\g)) = 0$.
\end{prop}

\begin{proof}
Consider equation \eqref{E4}, with $k\G \cdot v_\la = M_y$, where $y =
\Ind^{H \rtimes \G}_H k^\la$; thus $\dim_k y = |\G|$.
If $[Z_A(\la) : V_A(\mu)] > 0$ for some $\mu \in G$, then there is a
maximal vector $v_\mu = b v_\la$ in $Z(y) = (A \rtimes \G) v_\la$, for
some weight vector $b \in N_-$ (e.g., by Proposition \ref{P9}). We now
compute using equation \eqref{E4}, for each $z \in Z$:
\begin{equation}\label{E5}
\sum_{\g \in \G^\la} \la(\xi(z_\g)) b \cdot \g v_\la = b z v_\la = z
v_\mu = \sum_{\g \in \G^\mu} \mu(\xi(z_\g)) \g(b) \cdot \g v_\la.
\end{equation}

\noindent By the triangular decomposition, the coefficient of $\g v_\la$
vanishes if $\g \notin \G^\la \cap \G^\mu$ (or $\g(b) \notin k^\times
b$). Now apply symmetry and transitivity to get the result.
\end{proof}

\noindent The above proof does not use the information about $\g(b)
\notin k^\times b$; however, we will use it presently.\medskip

The main result in this subsection relates central characters to simple
subquotients of Verma modules.

\begin{theorem}\label{Tcenter}
$Z = \mf{Z}(A \rtimes \G)$ as above.
\begin{enumerate}
\item For all $\la \in G,\ \chi_\la$ is an algebra map $: Z \to k\G$.

\item To each $\la \in G$ is associated a subgroup $\G_\la$ of $\gsl$,
such that
\begin{itemize}
\item $\chi_\la(z) = \sum_{\g \in \G_\la} \la(\xi(z_\g)) \g\ \forall z =
\sum_{\g} z_\g \g \in Z$;

\item $\G_\la = \G_\mu$ if $\mu \in S^3_A(\la)$; and

%\item $Z$ acts on any maximal vector $v_\la$ in any $A \rtimes
%\G$-module, via: $z v_\la = \chi_\la(z) v_\la$;

\item given $\mu \in S^3_A(\la)$, $\chi_\la$ and $\chi_\mu$ are related
(on $Z$) via a character $\theta_{\la,\mu}$ of $\G_\la$: $\la(\xi(z_\g))
= \theta_{\la,\mu}(\g) \mu(\xi(z_{\g}))\ \forall \g \in \G_\la, z \in Z$.
\end{itemize}
\end{enumerate}
\end{theorem}

\begin{proof}
To show the first part, choose any $z,z' \in Z$ and $y = \Ind^{H \rtimes
\G}_H k^\la$. Then in the Verma module $Z(y) = (A \rtimes \G) v_\la$, by
Lemma \ref{Lcent},
\[ \chi_\la(zz') v_\la = \chi_\la(z'z) v_\la = z'z v_\la = z' \chi_\la(z)
v_\la = \chi_\la(z) \cdot z' v_\la = \chi_\la(z) \chi_\la(z') v_\la, \]

\noindent because $z'$ is central. The statement now follows, since $y$
is the regular representation of $\G$ (as a $\G$-module).\medskip

For the second part, we continue beyond the proof of Proposition
\ref{Pcenter}. Thus, $\chi_\la(z) = \sum_{\g \in \gsl} \la(\xi(z_\g))
\g$, and the ``improved" equation \eqref{E5} suggests that
$\la(\xi(z_\g)) = 0$ if $\g(b) \notin k^\times b$. Moreover,
$\la(\xi(z_\g)) = 0$ if and only if $\mu(\xi(z_\g)) = 0$.
Now denote $\G_{\la,\mu} := \{ \g \in \gsl : \g(b) \in k^\times b$ for
all $b$ such that $b v_\la \in Z(y)_\mu$ is maximal$\}$, and $\G_\la :=
\bigcap \G_{\nu',\nu''}$ for each $\nu' > \nu''$ in $S^3_A(\la)$ with
$[Z_A(\nu') : V_A(\nu'')] > 0$ (though a more suitable name would be
$\G_{S^3_A(\la)}$). The first two sub-parts are now obvious, and by
transitivity in  $S^3_A(\la)$, it suffices to show the last sub-part when
$[Z_A(\la) : V_A(\mu)] > 0$. But since $\G_\la$ acts on $k^\times \cdot
b$, this yields a character $\theta_{\la,\mu}$ of $\G_\la$; now look at
equation \eqref{E5} again, and we are done.
\end{proof}

\begin{cor}
If $\la \in G^\G$, then $\chi_\la$ is an algebra map $: Z \to
\mf{Z}(k\G)$.
\end{cor}

\begin{proof}
Use Proposition \ref{Pcent} and the first part of Theorem \ref{Tcenter}.
\end{proof}

\subsection{Another block decomposition}

We now show a block decomposition of $\calo$ using central characters,
under some extra assumptions.

\begin{theorem}\label{Tnoeth}\hfill
\begin{enumerate}
\item Suppose $k$ is algebraically closed if $|\G|>1$. Then $Z$ acts on
Verma modules $Z(x)$ (for $x \in X$) via a central {\em character}, i.e.,
an algebra map $: Z \to k$, say $\chi_x$. Then $\chi_x = \chi_{\la_x} =
\chi_{\g(\la_x)}$ on $Z \cap A\ \forall \g \in \G$.

\item Now suppose that $Z$ acts on every $Z(x)$ via an algebra map
$\theta_x : Z \to k$.
Then $\calo = \bigoplus_{\nu \in \Theta} \calo(\nu)$, where $\nu$ runs
over all distinct elements from among $\{ \theta_x : x \in X \}
\subset \hhh_{k-alg}(Z,k)$ (call this set $\Theta$), and $\calo(\nu)$ is
the full subcategory $\{ M \in \calo :\ \forall m \in M, z \in Z,\
\exists n \in \N \mbox{ such that } (z - \nu(z))^n \cdot m = 0 \}$.
\end{enumerate}
\end{theorem}

\begin{proof}
The first part follows from Propositions \ref{P10} and \ref{Pcent}, since
$M_x (\subset Z(x))$ is Schurian and any $z \in Z$ is an endomorphism of
$M_x$. For the second part, we use the following result:

\begin{lemma}\label{Lfit}
Given $\nu \neq \nu'$ in $\Theta$, $M \in \calo(\nu')$, and $z \notin
\ker(\nu - \nu')$, the map $\varphi := z - \nu(z)$ is an $A \rtimes
\G$-module isomorphism on $M$.
\end{lemma}

\noindent To see why, note that $\varphi = c + s$, where $s = z -
\nu'(z)$ is locally nilpotent on $M$, and $c = \nu'(z) - \nu(z) \in
k^\times$. Hence $\varphi$ is invertible.\medskip

Given $M \in \calo$, Proposition \ref{P8} implies (together with the
first part) that there exist $\nu_i \in \Theta$ and $n_i \in \N$ such
that $\prod_{i=1}^s (z - \nu_i(z))^{n_i}$ kills all of $M$, for all $z
\in Z$.
Now define $M(\nu) := \{ m \in M :\ \forall z \in Z, \exists n \in \N$
such that $(z - \nu(z))^n m = 0 \}$. Then this is an $A \rtimes
\G$-submodule. Using Lemma \ref{Lfit}, it is easy to see that each
$\calo(\nu)$ is closed in $A \rtimes \G - \Mod$ under taking submodules,
quotients, and extensions, and all maps between distinct blocks are
trivial. So if we show that $M = \bigoplus_\nu M(\nu)$, then each
$M(\nu)$ is a quotient of $M \in \calo$, hence also in $\calo$, hence in
$\calo(\nu)$, and we will be done. To do this, we need the following
standard lemma from commutative algebra.

\begin{lemma}
If $\{ \mf{m}_i : 1 \leq i \leq s \}$ are distinct maximal ideals in $Z$
for some $s \geq 1$, then $\sum_i \left( \prod_{j \neq i} \mf{m}_j
\right)^n = Z$ for all $n \geq 1$.
\end{lemma}

\noindent Use the lemma to write $1 \in Z$ as $1 = \sum_{i=1}^s r_i$,
with $r_i \in \left( \prod_{j \neq i} \ker \nu_j \right)^n$ and $n =
\max_i n_i$ ($s, \nu_i, n_i$ as above). Then any $m \in M$ equals $\sum_i
r_i m$, with $r_i m \in M(\nu_i)$. This shows that $M$ is the sum of its
components.

Finally, this sum is direct, for suppose $\sum_{i=1}^l m_i = 0$ for some
$l > 1$, with $m_i \in M(\nu_i)$ for each $i$. If $i<l$, pick $z_i \notin
\ker(\nu_i - \nu_l)$, and $n_i \in \N$ such that $(z_i -
\nu_i(z_i))^{n_i}$ kills $m_i$. Then $\prod_{i<l} (z_i -
\nu_i(z_i))^{n_i}$ kills $m_1, \dots, m_{l-1}$, whereas by Lemma
\ref{Lfit}, it is an isomorphism on $(A \rtimes \G) m_l \subset
M(\nu_l)$. Hence $m_l = 0$, and by induction on $l$, the other $m_i$'s
vanish as well.
\end{proof}

\subsection{The Conditions (S), linking, and central characters}

We now describe and relate different types of block decompositions. We
need some definitions. We note that is possible to obtain a block
decomposition of $\calo$ using any of these sets.

\begin{definition}\label{D3}\hfill
\begin{enumerate}
\item Define $CC(x) = S^4(x)$ to be the set of simple objects $\{ x' \in
X : \chi_{x'} = \chi_x \}$.

\item {\em Condition (S4)} holds for $\ag$, if all sets $S^4(x)$ are
finite.

\item We say that two indecomposable $A$-modules $M,N$ are {\em linked}
if\hfill\break
$\Hom(M,N) \neq 0$. Let the equivalence closure of $M$ under such a
relation be denoted by $[M]$.

\item Define $T(x)$ to be the equivalence closure of $x$ in $X$, under
the relations: $x \to F(x)$ and $x \to x'$ if $V(x') \in [V(x)]$.
\end{enumerate}
\end{definition}

\noindent (Note that we need the ``intermediate modules" between linked
modules to be indecomposable, otherwise any two modules are linked via:
$M \to M \oplus N \to N$.) Also recall $\sgx,\wt$ from Proposition
\ref{Pfunct}. We now compare these sets, and also mention a sufficient
condition when the center is ``nice" (this does indeed hold for wreath
products).

\begin{prop}\label{Pcentfunct}\hfill
\begin{enumerate}
\item Given a skew group ring $A \rtimes \G$, and $\la \in G, \g \in \G$,
\[ \g(CC_A(\la)) = \g(S^4_A(\la)) = S^4_A(\g(\la)) = CC_A(\g(\la)). \]

\item $\mf{Z}(\ag) \cap A = \mf{Z}(A)^{\G}$, and given $\la \in G$,
\[ \{ \mu \in G : \chi_\la \equiv \chi_\mu \mbox{ on } \mf{Z}(A)^{\G} \}
= \G(S^4_A(\la)). \]

\item Suppose $A$ is an integral domain, and each $\g \neq 1 \in \G$
nontrivially permutes the {\em restricted} (i.e., $G_0$-)weight spaces of
$A$. Then $\mf{Z}(A \rtimes \G) = \mf{Z}(A)^{\G} = e_{\G} \mf{Z}(A)
e_{\G}$.\medskip

\noindent Now suppose that $k$ is algebraically closed if $\G$ is
nontrivial.\medskip

\item Then we have:
\begin{eqnarray*}
S^3(x) & \subset & T(x) \cap \sgx\\
T(x) & \subset & S^4(x) \mbox{ if } Z = \mf{Z}(A)^{\G},\\
\wt^{-1}(\G(S^4_A(\la))) & \supset & \bigcup_{\wt(x) = \la} S^4(x).
\end{eqnarray*}

\item Condition (S4) holds for $\ag$ if it holds for $A$. If $\mf{Z}(\ag)
= Z = \mf{Z}(A)^{\G}$, then the converse is also true, since
\begin{equation}\label{Ecenter}
S^4(x) = \wt^{-1}(\G(S^4_A(\la_x)))\ \forall x \in X.
\end{equation}

\noindent Moreover, $(S4) \Rightarrow (S3)$ in this case.
\end{enumerate}
\end{prop}

\noindent Note that one of the parts implies that $\chi_\la = \la \circ
\xi\ \forall \la$, and if $k$ is algebraically closed, then $\chi_x =
\chi_{\la_x} = \theta_{\la_x}\ \forall x \in X$ (see Theorem
\ref{Tnoeth}).

\begin{proof}\hfill
\begin{enumerate}
\item If we prove: $CC_A(\g(\la)) \supset \g(CC_A(\la))\ \forall \g,\la$,
then
\[ \g(CC_A(\la)) = \g(CC_A(\g^{-1}(\g(\la)))) \supset
\g(\g^{-1}(CC_A(\g(\la)))) = CC_A(\g(\la)), \]

\noindent thereby proving the reverse inclusion.
To show the original inclusion, use Proposition \ref{Pcent}. If $\mu \in
CC_A(\la)$, then
\[ \chi_{\g(\mu)}(z) = \g \chi_\mu(z) \g^{-1} = \g \chi_\la(z) \g^{-1} =
\chi_{\g(\la)}(z) \]

\noindent for any central $z \in Z$, whence $\g(\mu) \in
CC_A(\g(\la))$.\medskip

\item First, $z \in A$ commutes with $A$ and with $\G$ if and only if $z
\in \mf{Z}(A)^{\G}$. Next, we prove both inclusions: if $\mu \in
S^4_A(\la)$ and $\g \in \G$, then $\chi_{\g(\mu)} \equiv \chi_\mu \equiv
\chi_\la$ on $\mf{Z}(A)^{\G}$.

For the reverse inclusion, take any $\mu \notin \G(CC_A(\la))$; we will
show that $\chi_\mu \neq \chi_\la$ on $\mf{Z}(A)^\G$. Since $\G$ is
finite, enumerate the {\it distinct central characters for $A$}
associated to $\G(CC_A(\mu)) \coprod \G(CC_A(\la))$ (equivalently from
above, $\G(\la) \coprod \G(\mu)$), as $\{ \chi_\mu, \chi_1, \dots, \chi_l
\}$. (In particular, $\chi_\la = \chi_i$ for some $i$, and so is
$\chi_{\g(\mu)}$ for all $\g(\mu) \notin CC_A(\mu)$.)

Now, $\ker \chi_\mu$ is a maximal ideal in $\mf{Z}(A)$, different from
each $\ker \chi_i$; this allows us to choose $z_i \in \mf{Z}(A)$ such
that $\chi_i(z_i) = 0 \neq \chi_\mu(z_i)$. Hence $z' = \prod_i z_i$
satisfies: $\chi_i(z') = 0\ \forall i,\ \chi_\mu(z') \neq 0$.

\noindent Finally, define $z := \sum_{\g \in \G} \g(z') \in
\mf{Z}(A)^\G$. Then if $\nu \in G$,
\[ \chi_\nu(z) = \sum_\g \chi_\nu(\g(z')) = \sum_\g \nu(\xi(\g(z'))) =
\sum_\g \g^{-1}(\nu(\xi(z'))) = \sum_\g \chi_{\g(\nu)}(z'), \]

\noindent so that from above, $\chi_\la(z) = \sum_\g \chi_{\g(\la)}(z') =
0$. On the other hand, $\chi_\mu(z)$ is a sum of zeroes and some
(positive) number of $\chi_\mu(z')$'s, which is nonzero (since $\chrc(k)
= 0$ if $|\G| > 1$). Hence $\chi_\mu \neq \chi_\la$ on $\mf{Z}(A)^\G$ if
$\mu \notin \G(CC_A(\la))$, as claimed.\medskip

\item The second equality comes from Lemma \ref{Lcenter}. We now prove
both inclusions for the first equality. Clearly, $\mf{Z}(A)^{\G} \subset
Z$,
%because it commutes with both $A$ and $\G$,
and conversely, let us claim\medskip

\noindent {\bf Claim.} $Z \subset A$.\medskip

\noindent (If this holds, then $Z = \mf{Z}(A)^{\G}$ by the previous
part.)
It remains to show the claim. Suppose $z = \sum_{\g \in \G} z_{\g} \g \in
Z$, with $z_{\g} \in A\ \forall \g$. Given $\g \neq 1$, we have to show
that $z_{\g} = 0$; now choose any $\la \in \Z \dd$ with $\g(\la) \neq
\la$ and $A_\la \neq 0$, and fix $0 \neq a_\la \in A_\la$. Then
\[ \sum_{\g} a_\la z_{\g} \g = a_\la z = z a_\la = \sum_{\g} z_{\g}
\g(a_\la) \g. \]

\noindent We now assume $z_{\g} \neq 0$, and obtain a contradiction.
Assume that $z_{\g} = a_{\theta_1} \oplus \dots \oplus a_{\theta_l}$ for
some weight vectors $0 \neq a_{\theta_j} \in A_{\theta_j}$ (with pairwise
distinct weights $\theta_j \in G$). Then $a_\la a_{\theta_j}$ is nonzero
(since $A$ is an integral domain) and in $A_{\la + \pi(\theta_j)}$, where
all (restricted) weights are in the {\it abelian} group $\Z \dd$, by the
RTA axioms. Similarly, $0 \neq a_{\theta_j} \g(a_\la) \in A_{\g(\la) +
\pi(\theta_j)}$.

So if $a_\la z = z a_\la$, then comparing the coefficient of $\g$ on both
sides, the sets of weights on both sides must be the same, whence their
sum is the same. Hence $\sum_j \pi(\theta_j) + l \la = \sum_j
\pi(\theta_j) + l \g(\la)$, or $l \la = l \g(\la)$, whence (in $\Z \dd$)
$\la = \g(\la)$, a contradiction.\medskip

\item We first show that if $[Z(x) : V(x')] > 0$, then $x$ and $x'$ are
linked. But this is clear, since $Z(x)$ is indecomposable: choose some $N
\subset M \subset Z(x)$ such that $M$ is indecomposable, and $M/N \cong
V(x')$. (This is possible by Lemma \ref{Lsummand}.) We now have $V(x')
\twoheadleftarrow M \hookrightarrow Z(x) \twoheadrightarrow V(x)$, so $x$
and $x'$ are linked. This proves that $S^3(x) \subset T(x)$.

Next, that $S^3(x) \subset \sgx$, follows from Proposition \ref{Pfunct}.

It thus remains to show that $T(x) \subset S^4(x)$. To see this, use
Theorem \ref{Tnoeth} again: on $Z \subset A$, $\chi_x \equiv \chi_{\la_x}
\equiv \chi_{\la_{F(x)}} \equiv \chi_{F(x)}$, so $F(x) \in S^4(x)\
\forall x$.
Otherwise if $x,x'$ are linked through a chain of indecomposable objects
in $\calo$, then there exists a unique ``central character block"
$\calo(\nu)$, that contains all these objects. In particular, since $z -
\nu(z)$ kills all simple objects in such a block, we are done.

Finally, we prove the third inclusion. Given $x' \in S^4(x)$ and $\la =
\la_x$, we note that $\chi_x \equiv \chi_{x'}$ on $\mf{Z}(\ag) = Z$.
Hence on $Z \cap A$, using Theorem \ref{Tnoeth}, we get: $\chi_{\la_x}
\equiv \chi_x \equiv \chi_{x'} \equiv \chi_{\la_{x'}}$, whence by a
previous part, $\la_{x'} \in \G(S^4_A(\la_x))$, as desired.\medskip

\item If $Z \subset A$, then by the previous part, $S^3(x) \subset T(x)
\subset S^4(x)$. Next, if $A$ satisfies (S4), then so does $\ag$ by the
previous part again, using Theorem \ref{T2}. We now only need to prove
that if $Z \subset A$, then equation \eqref{Ecenter} holds. But if $\mu
\in S^4_A(\la_x)$ and $x' \in \wt^{-1}(\mu)$, then $\chi_{\g(\mu)} \equiv
\chi_\mu \equiv \chi_\la$ on $\mf{Z}(A)^{\G}$. Now use Theorem
\ref{Tnoeth}; thus, $\chi_{x'} \equiv \chi_x$ on $Z$ - so $x' \in
S^4(x)$.
\end{enumerate}
\end{proof}

\subsection{Central characters for wreath products}

Finally, we come to the case of $R_{\mf{g}} = S_n \wr \mf{Ug}$ for
$\mf{g}$ a complex semisimple Lie algebra. The main result that we prove
here helps prove Proposition \ref{Pwreath}.

\begin{theorem}\label{Twreath}
Let $A \rtimes \G = (\mf{Ug})^{\otimes n} \rtimes S_n = R_{\mf{g}}$ (over
$k = \C$).
\begin{enumerate}
\item The center of $R_{\mf{g}}$ is $(\mf{Z}(\mf{Ug})^{\otimes n})^{S_n}
\cong \C[X_1, \dots, X_{ns}]$, with $s = \dim_{\C} \mf{h}$.

\item The center acts by the same central character on two simple objects
$V(x), V(x') \in \calo$, if and only if $\la_x \in S_n (W^n_{\mf{g}}
\bullet \la_{x'})$.

\item The sets of central characters of $A$ and $\ag$ are in bijection
with $(\mf{h}^*)^n / (W^n, \bullet)$ and $(\mf{h}^*)^{\oplus n} / (S_n
\wr W, \bullet)$ respectively. Every central character comes from a Verma
module.
\end{enumerate}
\end{theorem}

\noindent In particular, we obtain a central character block
decomposition, by Theorem \ref{Tnoeth} above. The rest of this section is
devoted to proving the theorem.

\begin{definition}
Given $\la \in G$, define $CC_A(\la) := S^4_A(\la) = \{ \mu \in G : \mu
\circ \xi = \la \circ \xi$ on $\mf{Z}(A) \}$.
\end{definition}

\begin{lemma}\label{Lcenter}
We work over any fixed ground field $k$.
\begin{enumerate}
\item Suppose a finite group $\G$ acts by algebra automorphisms on a
$k$-algebra $A$ (here, $|\G| \in k^\times$). Then the fixed point algebra
$A^{\G}$ is isomorphic to the spherical subalgebra $e_{\G} A e_{\G}$ as
subalgebras of $A \rtimes \G$, where $e_{\G} := \frac{1}{|\G|} \sum_{\g
\in \G} \g$.

\item If $A_1, \dots, A_n$ are $k$-algebras, then $\mf{Z}(\otimes_i A_i)
= \otimes_i \mf{Z}(A_i)$.

\item If $A_1, \dots, A_n$ are RTAs, and $\la_i \in G_i$ are weights,
then $\chi_{\la} = \otimes_i \chi_{\la_i}$ on $\mf{Z}(A)$, and $CC_A(\la)
= \times_i CC_{A_i}(\la_i)$, where $\la = (\la_1, \dots, \la_n)$.
\end{enumerate}
\end{lemma}

\begin{proof}\hfill
\begin{enumerate}
\item The map $: A^{\G} \to e_{\G} A e_{\G}$ ($a \mapsto e_{\G} a
e_{\G}$) is a $k$-algebra isomorphism.

\item One inclusion is clear; the proof of the other is by induction on
$n$, and the only (possibly) nontrivial step is to show it for $n=2$.
Given $k$-algebras $A,B$, suppose $z = \sum_i a_i \otimes b_i$ is
central, with the $b_i$'s linearly independent in $B$. Choose any $a \in
A$; then $az = za$ implies that all $a_i$'s are central, whence $z \in
\mf{Z}(A) \otimes B$. Now write $z = \sum_j a'_j \otimes b'_j$, where the
$a'_j$'s are now linearly independent in $\mf{Z}(A)$. Then $bz = zb$ for
all $b \in B$ implies that the $b'_j$'s are all central in $B$.

\item The statements make sense because of the previous part. It is easy
to show that the Harish-Chandra maps $\xi_i$ on $A_i$ and $\xi$ on $A$
satisfy: $\xi = \otimes_i \xi_i$ (extended by linearity). The first part
now follows.

One inclusion for the second part follows from the first part here; for
the other, by the previous part of this lemma, it suffices to start with
$\prod_{i=1}^n \chi_{\la_i}(z_i) = \prod_{i=1}^n \chi_{\mu_i}(z_i)$ for
all $i, z_i \in \mf{Z}(A_i)$. Now fix $i$ and set $z_j = 1$ for all $j
\neq i$; thus $\mu_i \in CC_{A_i}(\la_i)\ \forall i$.
\end{enumerate}
\end{proof}

It is not hard now, to compute the center of $R_{\mf{g}} = S_n \wr
\mf{Ug}$, or the corresponding blocks. We can now prove the main result
in this subsection.

\begin{proof}[Proof of Theorem \ref{Twreath}]\hfill
\begin{enumerate}
\item The first claim follows from Lemma \ref{Lcenter}, Proposition
\ref{Pcentfunct}, and the following two classical results from Lie
theory ($s = \dim_{\C} \mf{h}$ here):
\begin{enumerate}
\item $\mf{Z}(\mf{Ug}) \cong \C[X_1, \dots, X_s]$ (see \cite[Theorem
7.3.8]{Dix}).

\item If $G$ is a finite group acting on a finitely generated polynomial
algebra $\C[X_1, \dots, X_l]$ via reflections, then $\C[X_1, \dots,
X_l]^G \cong \C[Y_1, \dots, Y_l]$ (Chevalley's Theorem; see \cite{Che}).
\end{enumerate}

\noindent Applying these results (with $G = S_n$), the first part follows.

\item By Harish-Chandra's Theorem, ($CC_{\mf{Ug}}(\la) = W_{\mf{g}}
\bullet \la$. Now by Lemma \ref{Lcenter},)
$CC_{(\mf{Ug})^{\otimes n}}(\la_1, \dots, \la_n) = W^n_{\mf{g}} \bullet
(\la_1, \dots, \la_n)$. We are now done by Proposition \ref{Pcentfunct}
(and the previous part).

\item The only part not done above, involves computing {\it all} central
characters of $\ag$ (and not merely those in the Category $\calo$) - note
that the result itself implies that every central character corresponds
to some object in $\calo$. This remaining part follows from a special
case of the Nagata-Mumford Theorem (see e.g., \cite[Theorem 5.3]{Muk}).
\end{enumerate}
\end{proof}

Finally, we present the proof of an earlier, unproved result:

\begin{proof}[Proof of Proposition \ref{Pwreath}]
All but the first part (and that $\calo$ is finite length) were shown in
this section. However, by Harish-Chandra's theorem, $\mf{Ug}$ satisfies
all the Conditions (S). Thus, we are done by Theorem \ref{Tmain}.
Moreover, that $\mf{Ug}^{\otimes n} = \mf{U}(\mf{g}^{\oplus n})$ is of
finite representation type, is well-known.

Now consider a finite-dimensional simple $R_{\mf{g}}$-module; it is also
a finite-dimensional $\mf{Ug}^{\otimes n}$-module, hence is in
$\calo_{R_{\mf{g}}}$. So let us denote it by $V(x)$; say $\dim V(x) = d$.
Now by Corollary \ref{C1}, $\dim V_A(\la_x) | d$, and there are only
finitely many such $\la_x$'s. We are now done by Theorem \ref{T2}.
\end{proof}

\section{Each block is a highest weight category}

We now show that each block $\calo(x)$ has enough projectives, and is a
highest weight category, under some Condition (S). The following result
(see \cite[(A1)]{Don}) will be useful shortly; the proof is similar to
\cite[Theorem 3]{Kh}.

\begin{prop}\label{moreoncalo}\hfill
\begin{enumerate}
\item If $\Ext^1(Z(x),M)$ or $\Ext^1(M,F(Z(x)))$ is nonzero for $M \in
\calo$ and $x \in X$, then $M$ has a composition factor $V(x')$ with $x'
> x$.

\item If $X$ and $F(Y)$ have simple Verma flags, then $\Ext^1(X,Y) = 0$, and
\[ \dim_k(\hhh_A(X,Y)) = \sum_{x \in X} [X : Z(x)][Y : F(Z(F(x)))] \dim_k
\End_{\ag} V(x). \]
\end{enumerate}
\end{prop}

Given $x \in X$, we now define $\calo^{\leq x}$ (or $\calo^{<x}$) to be
the subcategory of objects $N \in \calo$, so that all simple subquotients
of $N$ are of the form $V(x')$ for $x' \leq x$ (or $x' < x$
respectively). Given $x,x' \in X$, define $\calo(x')^{\leq x} :=
\calo(x') \cap \calo^{\leq x}$, and similarly, $\calo(x')^{<x}$ - so
$Z(x) \in \calo(x)^{\leq x}$ and $Y(x) \in \calo(x)^{<x}$.\medskip

\noindent We first show that enough projectives exist in $\calo$.

\begin{lemma}\label{Lproj}
If $\ag$ satisfies Condition (S1), then $Z(x)$ is the projective cover of
$V(x)$ in $\calo(x)^{\leq x}$.
\end{lemma}

\begin{proof}
We know that $\calo(x)^{\leq x} \subset \calo(\la_x,1)$, so by
Proposition \ref{P8}, $Z(x) = P(x,1)$ is projective here. Moreover,
$Z(x)$ is indecomposable, with radical $Y(x)$. The usual Fitting Lemma
arguments now complete the proof.
\end{proof}

\begin{prop}\label{Pproj}
If $\ag$ satisfies Condition (S2), then each $\calo(x)$ has enough
projectives. If $\ag$ satisfies Condition (S3), then each block
$\calo(x)$ is equivalent to the category $\bfg$ of finitely generated
right modules over a finite-dimensional $k$-algebra $B = B_x$.
\end{prop}

\begin{proof}
The proof of the first part uses the $x$-component in the (block)
decomposition of some $P(y,l)$, as in \cite{BGG1}. The second part uses
the existence of progenerators and the usual Fitting lemma arguments, as
in \cite{Kh}.
\end{proof}

\noindent Let us denote the projective cover of $V(x)$ by $P(x)$.

\begin{prop}\label{Pskewformula}
For all $x \in X$ and $M \in \calo$, we have
\begin{eqnarray*}
\hhom(P(x),V(x')) & = & \delta_{x,x'} \End_{\calo} V(x),\\
\dim_k (\hhom(P(x), M)) & = & [M : V(x)] \dim_k \End_{\calo} V(x).
\end{eqnarray*}
\end{prop}

\begin{proof}
Both sides of the second equation are additive in $M$, over short exact
sequences. This reduces it to the case $M = V(x')$, i.e., the previous
equation, which holds by general properties of projective covers.
\end{proof}

To show that each block $\calo(x)$ is a highest weight category (see
\cite{CPS1} for the definition), we need a result from \cite{BGG1}; its
proof is also valid here.

\begin{prop}\label{Pfilt}
Recall what a $p$-filtration means, in Definition \ref{D2} above.
\begin{enumerate}
\item If $M \in \calo$ has a $p$-filtration, and $x \in X$ is maximal
(minimal) in the set of Verma subquotients $Z(x)$ of $M$, then $M$ has a
submodule (quotient) $Z(x)$, and the quotient (kernel of the quotient,
respectively) has a simple Verma flag.

\item Given $M_1, M_2 \in \calo$, $M = M_1 \oplus M_2$ has a
$p$-filtration if and only if each of $M_1$ and $M_2$ has a simple Verma
flag.
\end{enumerate}
\end{prop}

\begin{cor}\label{Cproj}
If $\ag$ satisfies Condition (S2), then every $P(x)$ has a
$p$-filtration, with one subquotient (the ``first" one) $Z(x)$, and all
others $Z(x')$ for some $x' > x$.
\end{cor}

\begin{proof}
This is because each $P(x,l)$ (see Proposition \ref{P8}) has a simple
Verma flag, with one subquotient $Z(x)$, and every other subquotient
$Z(x')$ for some $x' \geq x$. Now use Propositions \ref{Pfilt}, \ref{P8},
and \ref{Pskewformula}, as in \cite[Theorem 6]{Kh}. That $Z(x)$ is the
``highest" subquotient is as in Proposition \ref{Pfilt} (or from
\cite[(A3.1)(i)]{Don}).
\end{proof}

\begin{theorem}\label{Tskewhighest}
Each block $\calo(x)$ is a highest weight category if $\ag$ satisfies
Condition (S3).
\end{theorem}

\begin{proof}
We need Condition (S3) to ensure that the set of simple objects is
``interval-finite" with respect to the partial order. We are now done by
Corollary \ref{Cproj}.
\end{proof}

\section{BGG Reciprocity and the (symmetric) Cartan matrix}

\begin{stand}
For this section, $\ag$ satisfies Condition (S3), and if $\G$ is
nontrivial, we require that $k$ is algebraically closed.
\end{stand}

\begin{definition}
Fix a block $\calo(x)$, and order $S^3(x)$ such that $x_i \geq x_j
\Rightarrow i \leq j$. We then define the
\begin{itemize}
\item {\em decomposition matrix} $D_x$ by: $(D_x)_{ij} := [Z(x_i) :
V(x_j)]$.

\item {\em duality matrix} $F_x$ by: $(F_x)_{ij} := \delta_{x_i,
F(x_j)}$.

\item {\em Cartan matrix} $C_x$ by: $(C_x)_{ij} := [P(x_i) : V(x_j)]$.

\item {\em modified Cartan matrix} $C'_x$ by: $(C'_x)_{ij} := [P(x_i) :
V(F(x_j))]$.
\end{itemize}
\end{definition}

Note that the duality functor $F$ does not preserve each $V(x)$, so the
``usual" notion of the Cartan matrix is not symmetric, as in the
classical case of $\mf{Ug}$. However, a variant is.

\begin{prop}[BGG Reciprocity]\label{Precip}
For all $x,x' \in X$, we have
\begin{equation}\label{Erecip}
[P(x') : Z(x)] = [Z(F(x)) : V(F(x'))]
\end{equation}

\noindent and hence $C'_x$ is symmetric; more precisely, $C'_x = F_x
D^T_x F_x D_x F_x$.
\end{prop}

\noindent As a consequence, \cite[Theorem 9.1]{GGK} holds, reconciling
various notions of block decomposition; in particular, $S^3(x) = T(x)$
(recall Proposition \ref{Pcentfunct}).

\begin{proof}
Proposition \ref{P10} implies that $V(x)$ is Schurian now, for all $x$.
Hence using Propositions \ref{moreoncalo} and \ref{Pskewformula}, we get
that
\begin{eqnarray*}
[P(x') : Z(x)] & = & \dim_k \hhom(P(x'), F(Z(F(x)))) = [F(Z(F(x))) :
V(x')]\\
& = & [Z(F(x)) : F(V(x'))] = [Z(F(x)) : V(F(x'))].
\end{eqnarray*}

\noindent The second part is also standard, say, from the following
results.
\end{proof}

\begin{prop}
Let ${\rm Grot}_{\mathcal{D}}$ be the Grothendieck group of a category
$\mathcal{D}$.
\begin{enumerate}
\item $\groto = \oplus {\rm Grot}_{\calo(x)}$.

\item $F_x$ is symmetric and has order at most two.

\item $C_x = F_x D_x^T F_x D_x$. In particular, if $\G = 1$ then $C_x$ is
symmetric.

\item Each of the following sets is a $\Z$-basis for $\groto$:
\[ \{ [V(x)] : x \in X \},\ \{ [Z(x)] : x \in X \},\ \{ [P(x)] : x \in X
\}. \]
\end{enumerate}
\end{prop}

\begin{proof}
The first and last part are standard in a highest weight category; the
second part holds because $F(F(V(x))) = V(x)\ \forall x$; and the third
part follows from equation \eqref{Erecip} above.
\end{proof}

\section{The second setup - tensor products}\label{Sectensor}

We now look at the representation theory of tensor products of skew group
rings over regular triangular algebras. More precisely, we relate the
category \calo over such a product, to the respective categories
$\calo_i$ over each factor.

\subsection{Notation}

Fix $n \in \N$. We fix skew group rings $A_i \rtimes \G_i$ over RTAs
$A_i$, that satisfy the standing assumptions \ref{St1} and \ref{St2}.
Then (mentioned above) so does $\ag := \otimes_{i=1}^n (A_i \rtimes
\G_i)$, where $A = \otimes_i A_i,\ \G = \times_i \G_i$, and $G = \times_i
G_i$.

\subsection{Duality and tensor product decomposition}

We first mention some exact functors on tensor products of $\calo_j$'s.
For this subsection, we work with the setup mentioned in Section
\ref{Sduality} (on duality) above: we have associative $k$-algebras
$A'_j$, each containing a unital $k$-subalgebra $H'_j$; moreover, there
exist anti-involutions $i_j$ of each $A'_j$, that extend $\id_{H'_j}$.

We can now define $A' = \otimes_j A'_j$, and similarly, $H',i$, and the
Harish-Chandra categories $\calh'_j$ and $\calh' \supset \otimes_j
\calh'_j$. Objects in these categories all have ``formal characters", and
the (restricted) duality functor $F$ operates on each of these
categories. The following is now standard.

\begin{prop}\label{Pexact}
Fix $V_j \in \calh'_j$ for all $j$, and fix $1 \leq i \leq n$.
\begin{enumerate}
\item $F(\otimes_j V_j) \cong \otimes_j F(V_j)$, and $\ch_{\otimes_j V_j}
= \prod_j \ch_{V_j}$.

\item The functor $\tv : \calh'_i \to \calh'$, sending an $A'_i$-module
$M$ to the $A'$-module $\tv(M) := (\otimes_{j < i} V_j) \otimes M \otimes
(\otimes_{j > i} V_j)$, is exact.

\item As $A'_i$-modules, $\tv(M)$ is a direct sum of copies of $M$; more
precisely, $\tv(M) \cong M \otimes W$, for the vector space $W =
\otimes_{j \neq i} V_j$.
\end{enumerate}
\end{prop}

We now apply this in our setup. Define the Harish-Chandra and BGG
Categories $\calh$ (or $\calh_i$) and $\calo$ (or $\calo_i$)
respectively, for $A \rtimes \G \supset H$ (or $A_i \rtimes \G_i \supset
H_i$ respectively). Then the above result holds.

\begin{cor}\label{Cverma}
If $\otimes_i V_i(x_i) \cong \otimes_i V_i(x'_i)$ for $x_i, x'_i \in X_i$
for all $i$, then $x_i = x'_i\ \forall i$. Moreover, if $\y = \otimes_i
y_i$, then $Z(\y) \cong \otimes_i Z_i(y_i)$.
\end{cor}

\begin{proof}
For the first part, apply Proposition \ref{Pexact} to both sides (having
first fixed an $i$). Thus, the left side is a direct sum of copies of
$V_i(x_i)$, and similarly for the other side. Now apply Lemma
\ref{Lsummand} (for $R = A_i \rtimes \G_i$); thus $V_i(x_i) \cong
V_i(x'_i)$ is the only simple summand on both sides, and we are done.

For the second part, we note that $\times_i Y_i \subset Y$, so if we
define $\y$ as above, then $Z(\y) \to \otimes_i Z_i(y_i) \to 0$.
Moreover, properties of induction functors (and the triangular
decomposition) imply that we can compare their formal characters:
\[ \ch_{Z(\y)} = \ch_{B_-} \ch_{\y} = \ch_{\otimes_i B_{i,-}}
\ch_{\otimes_i y_i} = \prod_i \ch_{B_{i,-}} \ch_{y_i} = \prod_i
\ch_{Z_i(y_i)}, \]

\noindent whence the two must be isomorphic.
\end{proof}

\section{Complete reducibility}

We now show that all notions of complete reducibility (i.e., in the four
setups at the start, and always only for finite-dimensional objects in
$\calo$) are equivalent. We need a small result first, since two of the
parts below are similar. For this section, we do not need $k$ to be
algebraically closed.

\begin{prop}\label{Pcr}
Suppose $A'$ and $A$ are RTAs, and $A \rtimes \G$ is a skew group ring
satisfying the Standing Assumptions \ref{St1} and \ref{St2}. Also say
there is a finite-dimensional vector space $U$ and $0 \neq u \in U$, such
that
\begin{enumerate}
\item $T : M \mapsto U \otimes_k M$ is an exact covariant functor $:
\calo_{A'} \to \calo_{A \rtimes \G}$.

\item If $v_{\mu'}$ has highest weight in the $A'$-Verma module
$Z_{A'}(\mu')$, then $v_{\mu'} \otimes u$ has highest weight in
$T(Z_{A'}(\mu'))$, which is also standard cyclic.

\item The map: $- \otimes u$ takes weight spaces to weight spaces, and
the induced map $: G_{A'} \to G_A$ is compatible with the partial orders
on $G_{A'}, G_A$.
\end{enumerate}

\noindent Then if complete reducibility holds for finite-dimensional
modules in $\calo_{A \rtimes \G}$, the same holds in $\calo_{A'}$.
\end{prop}

\begin{proof}
We will show that every short exact sequence between simple
finite-dimensional objects $0 \to V(\la') \to V \to V(\mu') \to 0$ in
$\calo_{A'}$ splits. Now assume that there is some such nonsplit
sequence. We may assume (using Proposition \ref{Pext} with $|\G| = 1$,
and) using the duality functor $F$ if necessary, that $\mu' > \la'$. Then
$V$ is standard cyclic; say $v_{\mu'}$ spans its $\mu'$-weight space.

Now apply $T$ to the sequence; by assumption, we get a short exact
sequence in $\calo_{A \rtimes \G}$, and each object is finite-dimensional
since $U$ is. On the one hand, $T(V)$ is standard cyclic, and generated
by $v_{\mu'} \otimes u$ (by assumption). On the other hand, the short
exact sequence in $\calo_{A \rtimes \G}$ splits, and by assumption,
$v_{\mu'} \otimes u$ has higher weight than the weights for $T(V(\la'))$
- whence it must lie in the complement to $T(V(\la'))$. In particular, it
cannot generate the entire module $T(V)$, a contradiction.
\end{proof}

We now prove the equivalence of complete reducibility in the four setups
(we do this in two stages).
We assume that all these setups involve (skew group rings over) RTAs,
satisfying Standing Assumptions \ref{St1} and \ref{St2}, but not
necessarily any of the Conditions (S).

\begin{theorem}\label{Tssskew}
Given such a skew group ring $A \rtimes \G$, complete reducibility holds
(for finite-dimensional modules) in $\calo_A \subset A$-{\em mod}, if and
only if it holds in $\calo_{A \rtimes \G}$.
\end{theorem}

\begin{proof}
First note that the existence of finite-dimensional (simple) modules in
$\calo_A$ and in $\calo_{A \rtimes \G}$ is equivalent, by Theorems
\ref{T2} and \ref{Tfacts}.
Now suppose that complete reducibility holds in $\calo_A$. Apply
Proposition \ref{Pss} (using Lemma \ref{LO} above), to the abelian
subcategories $\calp,\cald$ of $H$-semisimple finite-dimensional modules
inside $\calo_A, \calo_{A \rtimes \G}$ respectively. This proves one
implication.

Conversely, apply Proposition \ref{Pcr}, with $A' = A, U = k\G, u = 1 \in
U$. (Then $T = \Ind^{A \rtimes \G}_A$ is exact, and the other assumptions
also hold; for instance, $T(V_A(\la)) = V(y)$, where $y = \Ind^{H \rtimes
\G}_H k^{\la}$ for $\la \in G$.)
\end{proof}\medskip

Now suppose we are working with skew group rings (as above) $A_i \rtimes
\G_i$, and we define $A := \otimes_i A_i, \G := \times_i \G_i$.

\begin{theorem}\label{Tsstensor}
Complete reducibility holds in $\calo$ if and only if it does so in all
$\calo_i$.
\end{theorem}

\begin{proof}
We point out that we will use previous results, as well as Proposition
\ref{P3} (below) in the setting $|\G| = 1$.

We now show the result. First, by Theorem \ref{Tssskew}, it is enough to
check this for $|\G| = 1$. Hence $\calo_i = \calo_{A_i}$ and $\calo =
\calo_A$. Next, the existence of finite-dimensional modules in $\calo$ is
equivalent to that in $\calo_i$ for all $i$; this follows from
Proposition \ref{P3} (when $|\G| = 1$).\medskip

Therefore we now assume that there exist finite-dimensional modules in
each $\calo_i$ (and in $\calo$). Suppose complete reducibility holds in
each $\calo_i$, and we have a non-split short exact sequence $0 \to
V(\la) \to V \to V(\mu) \to 0$ of finite-dimensional modules (i.e., an
indecomposable module of length 2).
From Proposition \ref{Pext}, $\la > \mu$ or $\la < \mu$ if the sequence
does not split; using the duality functor $F$ if necessary, assume that
$\la < \mu$.

Fix an $i$ such that $\mu_i > \la_i$; now by assumption, the sequence
does split as finite-dimensional $A_i$-modules. Suppose $V = V(\la)
\oplus M$ in $\calo_i$. By Proposition \ref{Pext}, $V$ is a standard
cyclic $A$-module; say $V = A v_\mu$. Then by $H_i$-semisimplicity and
Proposition \ref{Pexact}, $v_\mu$ is in the $H_i$-weight space $V_{\mu_i}
= V(\la)_{\mu_i} \oplus M_{\mu_i} = M_{\mu_i}$ (since $\la_i < \mu_i$,
and using Proposition \ref{P3} below, for $|\G| = 1$).
But $M \cong V(\mu)$ is a direct sum of copies of $V_i(\mu_i)$. Hence so
is $A_i v_\mu \subset M$, and hence also, $V = A v_\mu = (\otimes_{j \neq
i} A_j) \cdot A_i v$ - hence, its $A_i$-submodule $V(\la) \cong \bigoplus
V_i(\la_i)$ (as $A_i$-modules) as well - and we get a contradiction by
Lemma \ref{Lsummand}.

Hence all extensions with $\la < \mu$ split, and by duality, so do all
extensions with $\la > \mu$. Thus complete reducibility holds in
$\calo$.\medskip

Conversely, fix $i$, and for all $j \neq i$, fix simple
finite-dimensional modules $V_j(\la_j) \in \calo_j$ (these exist by above
remarks). Define the functor $\tv : \calo_i \to \calo$ as in Proposition
\ref{Pexact} above; thus, $\tv$ is exact.

Now apply Proposition \ref{Pcr} with $A' = A_i, |\G| = 1$, $A$ as above,
$U = \otimes_{j \neq i} V_j(\la_j)$ (so $T = \tv$), and $u = \otimes_{j
\neq i} v_{j,\la_j}$, the unique (up to scalars) highest weight vector in
$U$. (Note that $\tv(V_i(\la_i)) = V_A(\la_1, \dots, \la_n)$ by
Proposition \ref{P3} again.)
\end{proof}

\section{Condition (S) for tensor products}

The idea now, is to relate the Categories $\calo_i$ for $A_i \rtimes
\G_i$, to $\calo = \calo_{A \rtimes \G}$. We need to characterize the
simple objects in the latter, in terms of those in the former categories.
As above, we have the various sets $X_i$ ($X$) of simple $H_i \rtimes
\G_i$- ($H \rtimes \G$-) modules that are $H_i$- (respectively $H$-)
semisimple.

\begin{stand}
For this section and the next, assume that Standing Assumption \ref{St3}
holds.
\end{stand}

\begin{theorem}\label{Tsimples}
$X = \times_j X_j$.
\end{theorem}

\noindent The example one should have in mind, is the case $|\G| = 1$;
then $X_j = G_j = \hhh_{k-alg}(H_j, k)$, and $X = G = \times_j G_j$. Also
note that if $\x \in X$, then $\la_{\x} = (\la_{x_1}, \dots, \la_{x_n})
\in \times_j G_j$.

\begin{proof}
For this, we need the following ``general" result\footnote{I thank Mitya
Boyarchenko for telling me this general result.}.

\begin{prop}\label{Pmitya}
Let $k$ be an algebraically closed field; all tensor products are over
$k$. Let $R_i$ be unital $k$-algebras for $1 \leq i \leq n$, and define
$R := \otimes_i R_i$. If $P_i$ (or $P$) is the set of (isomorphism
classes of) finite-dimensional simple modules over $R_i$ (or $R$,
respectively) (for each $i$), then the map $\otimes : \times_i P_i \to
R-\Mod$, taking $([M_1], \dots, [M_n])$ to $\left[ \otimes_i M_i
\right]$, is a bijection onto $P$.
\end{prop}

\noindent The proof of this result is an exercise in Wedderburn theory,
keeping in mind that if a simple $A$-module $M$ is finite-dimensional
(over $k$, where $A$ is a $k$-algebra), then $A$ acts on $M$ via a
subalgebra $A' \subset \mf{gl}_k(M)$. (So $A'$ is Artinian.) Moreover,
$M$ is a faithful simple $A'$-module - which makes $A'$ a simple algebra,
hence of the form $\End_k(k^l)$.\medskip

We now prove the theorem. If $|\G| = 1$ then the result is clear, since
$G = \times_j G_j$. If not, then $k$ is algebraically closed of
characteristic zero, by assumption. By Proposition \ref{Pmitya}, we only
need to show that if $M_{\x} := \otimes_i M_{x_i}$, then the $M_{x_i}$
are $H_i$-semisimple if and only if $M_{\x}$ is $H$-semisimple. The
``only if" part is clear, and for the ``if" part, it is not hard to show
that for each $\la \in G$, $(M_{\x})_\la \subset \otimes_i
(M_{x_i})_{\la_i}$. Therefore
\[ M_{\x} = \bigoplus_{\la \in G = \times_i G_i} (M_{\x})_\la \subset
\bigoplus_{\la \in G} \bigotimes_i (M_{x_i})_{\la_i} = \bigotimes_i
\bigoplus_{\la_i \in G_i} (M_{x_i})_{\la_i} \subset \otimes_i M_{x_i} =
M_{\x} \]

\noindent whence all inclusions are equalities, and $M_{x_i} =
\bigoplus_{\la_i \in G_i} (M_{x_i})_{\la_i}\ \forall i$.
\end{proof}

\noindent It is now easy to determine the simple objects in $\calo$ (and
their characters in terms of those of the simple objects in $\calo_i$):

\begin{prop}[``Weyl Character Formula 2"]\label{P3}
$V \in \calo$ is simple if and only if  $V = \otimes_i V_i(x_i)$ for some
simple $V_i(x_i) \in \calo_i$. (Moreover, $\ch_V = \prod_i
\ch_{V_i(x_i)}$.)
\end{prop}

\begin{proof}
Fix $i$. Given $x_j \in X_j$ for all $j$, $V := \otimes_j V_j(x_j)$ is
isomorphic to $V_i(x_i) \otimes W$ (as $A_i \rtimes \G_i$-modules) by
Proposition \ref{Pexact} - where $W$ is the vector space $\otimes_{j \neq
i} V_j(x_j)$.
In particular, any maximal vector in $V$ generates an $A_i \rtimes
\G_i$-submodule; this submodule must also be a direct sum of copies of
$V_i(x_i)$, whence the vector has $i$th weight component $\g_i(\la_i)$
for some $\g_i \in \G_i$.

This holds for every $1 \leq i \leq n$; thus any maximal vector in $V$
has weight $\g(\la)$, whence it is in $\otimes_i M_{x_i} = M_{\x}$. But
$M_{\x}$ is simple by Theorem \ref{Tsimples}, so the standard cyclic
module $\otimes_j V_j(x_j)$ is simple, whence it equals $V(\x)$.

Thus, the map $\Psi : \times_i X_i \to X$ (taking $(V_i(x_i))_i$ to
$\otimes_i V_i(x_i)$) is surjective. By Corollary \ref{Cverma} above,
$\Psi$ is also injective; hence we are done.
\end{proof}

\begin{cor}\label{Ctensor}
$F(\x) = (F(x_j))_j$ in $\calc$, and $F(V(\x)) = \otimes_j F(V_j(x_j))$
in $\calo$.
\end{cor}

\begin{proof}
The statement in $\calc$ follows from Proposition \ref{Pexact}, and the
other equation now follows by using Propositions \ref{P3} and \ref{P4}.
\end{proof}\medskip

We now relate Condition (S3) for $A \rtimes \G$, with the same condition
for all $A_i \rtimes \G_i$. Define the sets $S'_i(x_i) \subset S^3_i(x_i)
\subset X_i$ (and $S'(\x) \subset S^3(\x) \subset X$) as in Definition
\ref{D1} above.

\begin{theorem}\label{Ptensor}\hfill
\begin{enumerate}
\item For each $\x,\ Z(\x)$ has finite length if and only if each
$Z_j(x_j)$ does.

\item $\calo$ is finite length if and only if all $\calo_i$'s are finite
length.

\item Given $x_j \in X_j$ for all $j$, $S'(\x) \subset \times_j S'_j(x_j)
\subset S^3(\x) \subset \times_j S^3_j(x_j)$.
\end{enumerate}
\end{theorem}

\begin{proof}
We will need the following elementary result.

\begin{lemma}\label{Lelem}
Given a ring $R_j$ and an $R_j$-module $M_j$ for each $1 \leq j \leq n$,
define $R := \otimes_j R_j$ and $M := \otimes_j M_j$. If each $M_j$ has a
chain of submodules
$M_j = M_{j,0} \supsetneq M_{j,1} \supsetneq \dots \supsetneq M_{j,l_j}
= 0$,
then $M$ has a chain of length $\prod_j l_j$, with set of subquotients
$\{ \otimes_j \left( M_{j, i_j-1} / M_{j, i_j} \right) : 1 \leq i_j \leq
l_j\ \forall j \}$.
\end{lemma}

\begin{enumerate}
\item If $l_j = l(Z_j(x_j))$, then $Z(\x)$ has length $\prod_j l_j$ by
Lemma \ref{Lelem} and Proposition \ref{P3}.
Conversely, suppose some $Z(\x)$ has finite length, say $n$, but some
$Z_i(x_i)$ does not. Then we can construct an arbitrarily long filtration
of $Z_i(x_i)$, say of length $> n$. By Lemma \ref{Lelem}, this gives a
filtration of $Z(\x)$ of length larger than $n$, a contradiction.\medskip

\item This follows from Proposition \ref{Pfinlength} and the previous
part.\medskip

\item This is done in stages.\medskip

\noindent {\bf Step 1.}
We show the first inclusion (and will use it later, to show the third
inclusion). We are to show that if $\x' \in \times_j S'_j(x_j)$, and
$[Z(\x') : V(\x'')]$ or $[Z(\x'') : V(\x')]$ is nonzero, then $\x'' \in
\times_j S'_j(x_j)$.

Suppose $[Z(\x') : V(\x'')] > 0$ (the proof is similar in the other
case). Fix $i$. As $A_i \rtimes \G_i$-modules, $Z(\x') = \otimes_j
Z_j(x'_j)$ (or $V(\x'') = \otimes_j V_j(x''_j)$) is isomorphic (by above
results) to a direct sum of copies of $Z_i(x'_i)$ (or $V_i(x''_i)$,
respectively).
Thus, $V_i(x''_i)$ is a subquotient of $V(\x'')$ (as $A_i \rtimes
\G_i$-modules), hence also of $Z(\x') = \oplus Z_i(x'_i)$. Now use Lemma
\ref{Lsummand} - so $x''_i \in S'_i(x'_i) = S'_i(x_i)$ (by transitivity)
for all $i$.\medskip

\noindent {\bf Step 2.}
Since in a Cartesian product of graphs, the connected components are the
products of connected components in each factor, the second inclusion
holds if we show that if $\x' \in S^3(\x)$ and $x''_i \in S'_i(x_i)$ for
some $i$, then $\x'' := (x'_1, \dots, x'_{i-1}, x''_i, x'_{i+1}, \dots,
x'_n)$ is in $S^3(\x)$. By transitivity, this further reduces to showing
the same when $[Z_i(x'_i) : V_i(x''_i)]$ or $[Z_i(x''_i) : V_i(x'_i)]$ is
nonzero.

We show the proof in the second case (the proof in the first case is
similar). Apply Lemma \ref{Lelem} to the chains of submodules
\[ M_j = Z_j(x'_j) \supsetneq Y_j(x'_j) \supset 0 \ \forall j \neq i,
\qquad M_i = Z_i(x''_i) \supset M' \supsetneq N' \supset 0, \]

\noindent where $M'/N' \cong V_i(x'_i)$. Then $[Z(\x'') : V(\x')] > 0$,
whence $\x'' \in S^3(\x') = S^3(\x)$ (by transitivity).\medskip

\noindent {\bf Step 3.}
Finally, the third inclusion follows from Step 1 and Corollary
\ref{Ctensor}, since $\times_j S^3_j(x_j)$ is now closed under both
operations.
\end{enumerate}
\end{proof}

\noindent Thus, the last part is a step towards showing the equivalence
of Condition (S3) in the two setups. In fact, it is enough in the case
that we need:

\begin{cor}\label{Cprod}
If $|\G| = 1$, then $S^3_A(\la_1, \dots, \la_n) = \times_j S^3_j(\la_j)\
\forall i,\la_i \in G_i$; hence Condition (S3) holds for $A$ if and only
if it holds for every $A_i$.
\end{cor}

\begin{proof}
If $|\G| = 1$, $X = G$, $X_i = G_i\ \forall i$, and $S'(\la) = S^3(\la)\
\forall \la$ (since $F(\la) = \la$); thus, all inclusions in the final
part of the above result are equalities.
\end{proof}

We conclude this section by simultaneously doing two things. First, the
last part of Theorem \ref{Ptensor} suggests that some generalization of
the formula in Corollary \ref{Cprod} is possible. Moreover, we wish to
complete the following ``commuting" cube (relating the ``$S^3$-sets" in
various setups), given $\g \in \G$, $\la \in G$, and $\x \in X$ such that
$\la = \la_{\x}$ (also see Proposition \ref{Pfunct}); note that $F$
preserves every (known) vertex.
\begin{equation}\label{Ediag2}
\xymatrix{
  & \{ S^3_i(\g(\la_i)) \}
     \ar[rr]^{\rtimes}
     \ar[dd]^{\times}
  && \{ S^3_i(x_i) \}
     \ar[dd]^{\times}\\
  \{ S^3_i(\la_i) \}
     \ar[rr]^{\rtimes}
     \ar[dd]^{\times}
     \ar[ru]^{\g(\cdot)}
  && \{ S^3_i(x_i) \}
     \ar[dd]^{\times}
     \ar[ru]^{\g(\cdot) = \id}\\
  & S^3_A(\g(\la))
     \ar[rr]^{\rtimes}
  && \bigcirc\hspace*{-1.7ex}?\\
  S^3_A(\la)
     \ar[rr]^{\rtimes}
     \ar[ru]^{\g(\cdot)}
  && \bigcirc\hspace*{-1.7ex}?
     \ar[ru]^{\g(\cdot) = ?}\\
}
\end{equation}

To do this, we introduce the following notation. Define $F^0(x) := x$ and
$F^1(x) := F(x)$; for all $\vi = (\vi_1, \dots, \vi_n) \in (\Z / 2
\Z)^n$, define $F^\vi (\x) := (F^{\vi_1}(x_1), \dots, F^{\vi_n}(x_n))$.
For all $\vi,\vi'$, we thus have: $F^\vi(F^{\vi'}(\x)) = F^{\vi +
\vi'}(\x)$.

Now note that the cube above is identical to the one in diagram
\eqref{Ediag3} above, with the $Z$'s replaced by $S^3$. However, the
missing set of objects $\bigcirc\hspace*{-1.7ex}?$ (which should be
contained in $\times_i S^3_i(x_i)$) {\it cannot} just be $S^3(\x)$, since
it could also have been $S^3(F^{\vi}(\x))$ in its place. So the natural
candidate now, would be the finite union $\bigcup_{\vi}
S^3(F^{\vi}(\x))$. The question now is: when is this set closed under all
the $F^{\vi}$? We present a sufficient condition.

\begin{theorem}\label{Ttensor}
If $F(S'_i(x_i)) = S'_i(F(x_i))\ \forall i,x_i$, then
\begin{equation}\label{E3}
\times_j S^3_j(x_j) = \bigcup_{\vi \in \mathbb{P}(\Z / 2 \Z)^n}
S^3(F^\vi(\x)).
\end{equation}
\end{theorem}

\begin{remark}\hfill
\begin{enumerate}
\item Here, $\mathbb{P}(\Z / 2 \Z)^n$ denotes the quotient of $(\Z / 2
\Z)^n$ by the diagonal copy of $\Z / 2 \Z$ sitting in it as $\{ (0,
\dots, 0), (1, \dots, 1) \}$ (we abuse notation). This is because by
Corollary \ref{Ctensor}, $F(\x) = F^{(1, \dots, 1)}(\x) \in S^3(\x)$.

\item In general, it is not clear that $F(S'(x)) = S'(F(x))$. In the
degenerate case $\G = 1$, this holds because $F(\la) = \la$ for any RTA
$A$ and any $\la \in G$. (So $F(S'(\la)) = S'(\la) = S^3(\la)$.) 
\end{enumerate}
\end{remark}

\begin{proof}
Note that $S^3(F^\vi(\x)) \subset \times_j S^3_j(F^{\vi_j}(x_j)) =
\times_j S^3_j(x_j)\ \forall \vi$. Moreover, Theorem \ref{Ptensor}
implies that we can cover all of $\times_j S^3_j(x_j)$ from any fixed
$\x$, using the duality functors $F^\vi$, and the relation of being a
simple subquotient of a Verma module.\medskip

\noindent {\bf Claim.} $F^\vi(S^3(\x)) = S^3(F^\vi(\x))\ \forall
\vi$.\medskip

\noindent Before we show the claim, let us use it to prove the theorem.
Note, by the first paragraph in this proof, that the right-hand side of
equation \eqref{E3} is contained in the left side. Moreover, both sides
are closed under the ``Verma-to-simple relation", as well as all possible
$F^\vi$'s (by the claim). But both sides also partition $X$ (since $X_j =
\coprod S^3_j(x_j)$ for all $j$), whence all inclusions are now equalities,
as desired.

It remains to prove the claim. Moreover, it suffices to show, for any
coordinate vector $e_i$, that
$F^{e_i} (S^3(\x))$ $\subset S^3(F^{e_i}(\x))$, for then we get that
\[ S^3(F^{e_i}(\x)) = F^{e_i}(F^{e_i}(S^3(F^{e_i}(\x)))) \subset
F^{e_i}(S^3(\x)) \subset S^3(F^{e_i}(\x)) \]

\noindent from above, whence all inclusions become equalities. The
statement for general $\vi$ follows by a series of compositions of
various $F^{e_j}$'s.

We now conclude the proof. Say $\x' \in S^3(\x)$ satisfies $F^{e_i}(\x')
\in S^3(F^{e_i}(\x))$, and $[Z(\x') : V(\x'')] > 0$ (the case when
$[Z(\x'') : V(\x')] > 0$ is similar). By the proof of the previous
theorem, $[Z_j(x'_j) : V_j(x''_j)] > 0\ \forall j$. Hence by the given
assumption,
\[ F^{e_i}(\x'') \in F^{e_i} \left( \times_j S'_j(x'_j) \right) =
S'_i(F(x'_i)) \times \left( \times_{j \neq i} S'_j(x'_j) \right) \subset
S^3(F^{e_i}(\x')), \]

\noindent where the last inclusion follows from Theorem \ref{Ptensor}.
But then by choice of $\x'$, we have $F^{e_i}(\x') \in S^3(F^{e_i}(\x))$,
so $F^{e_i}(\x'') \in S^3(F^{e_i}(\x))$ too.

The other relation is that of duality. But if $\x' \in S^3(\x)$ satisfies
$F^{e_i}(\x') \in S^3(F^{e_i}(\x))$, then $F^{e_i}(F(\x')) =
F(F^{e_i}(\x')) \in S^3(F^{e_i}(\x))$, since this latter set is also
closed under duality ($F = F^{(1,\dots,1)}$).

Hence the closure under both these relations (of $\x$, which does
satisfy: $F^{e_i}(\x) \in S^3(F^{e_i}(\x))$) also satisfies the same
property. Thus, $F^{e_i}(S^3(\x)) \subset S^3(F^{e_i}(\x))$ for all $i$,
which finishes the proof of the claim.
\end{proof}

\section{Functoriality of the BGG Category: combining the
setups}\label{Sfinal}

Combining all of the above analysis, we see that we have four setups in
which it makes sense to consider the BGG Category $\calo$, as given in
\eqref{E1}. We now sketch a proof of Proposition \ref{Pdiag} above:

That the first diagram commutes, is proved in Theorem \ref{Tsimples}; for
the second diagram, use Proposition \ref{P3} and Corollary \ref{Ctensor}.
The commuting cube in the $\calc$'s uses the first diagram in Proposition
\ref{Pdiag}, Proposition \ref{P4}, Theorem \ref{Tsimples}, and Corollary
\ref{Ctensor}.
Finally, the result on Verma modules uses Remark \ref{R4} and Corollary
\ref{Cverma}. (The cube mentioned immediately after Proposition
\ref{Pdiag}, also uses these results.)\medskip

We now show the main theorems, after stating a part of Theorem
\ref{Tmain} in full detail:

\begin{prop}\label{Pequiv}
Given $\la_i \in G_i$ for all $i$, the following are equivalent:
\begin{enumerate}
\item $V_{A_i}(\la_i)$ is finite-dimensional for all $i$.

\item $V_A(\la)$ is finite-dimensional, where $\la = (\la_1, \dots,
\la_n)$.

\item $V_{A_i \rtimes \G_i}(x_i)$ is finite-dimensional for all $i$ and
any $x_i$ with $\la_i = \la_{x_i}$.

\item $V_{A \rtimes \G}(\x)$ is finite-dimensional for any $\x$ with $\la
= \la_{\x}$.
\end{enumerate}

\noindent and the following are equivalent:

\begin{enumerate}
\item $Z_{A_i}(\la_i)$ has finite length for all $i$.

\item $Z_A(\la)$ has finite length, where $\la = (\la_1, \dots, \la_n)$.

\item $Z_{A_i \rtimes \G_i}(x_i)$ has finite length for any $i$ and $x_i$
with $\la_i = \la_{x_i}$.

\item $Z_{A \rtimes \G}(\x)$ has finite length for any $\x$ with $\la =
\la_{\x}$.
\end{enumerate}

\noindent In all cases, $V$ and $Z$ stand for simple and Verma
objects in the appropriate Category $\calo$, respectively.
\end{prop}

\begin{proof}
The first set of equivalences follows from Proposition \ref{P3} (for the
vertical arrows in diagram \eqref{Ediag1}) and Theorem \ref{Tfacts} (for
the horizontal arrows). The second set of equivalences follows from
Theorems \ref{Tfacts} and \ref{Ptensor}.
\end{proof}

Finally, we ``collect together" a proof of the main results of this
article.

\begin{proof}[Proof of Theorem \ref{Teqs}]
The first part is Theorem \ref{Tsimples}; the second comes from
Propositions \ref{Pfunct} and \ref{Pcentfunct}. For the third part, for
$m=3$ and 4, use Corollary \ref{Cprod} and Lemma \ref{Lcenter}
respectively. Given the result for $m=3$, apply $\pi = \times_i \pi_i$ to
get it for $m=2$; moreover, this implies the $m=1$ statement because of
the partial order on $G$.
\end{proof}

\begin{proof}[Proof of Theorem \ref{Tmain}]\hfill
\begin{enumerate}
\item The equivalence of complete reducibility holding in the four
setups, follows from Theorems \ref{Tssskew} and \ref{Tsstensor}. The
second and third parts are shown above. The fourth part (about $\calo$
and Verma modules being of finite length) follows from Proposition
\ref{Pfinlength}. Finally, the equivalences of the Conditions (S) across
the four setups follow from Theorems \ref{Tconditions} and \ref{Teqs} and
Proposition \ref{Pcentfunct}.

\item This follows from Theorem \ref{Tconditions} and Proposition
\ref{Pcentfunct}.
\end{enumerate}
\end{proof}

\begin{proof}[Proof of Theorem \ref{Timplies}]
The first part follows from Theorem \ref{Tconditions}, and the final two
parts from Theorem \ref{Tskewhighest}, Corollary \ref{Cproj}, and
Proposition \ref{Pproj}.
\end{proof}

\appendix
\section{Application: symplectic oscillator algebras}

\subsection{Definitions}

Now fix $A = H_f$ for a polynomial $f$ (see \cite{Kh}) and $n \in \N$.
(Note that $H_f$ is an {\em infinitesimal Hecke algebra over}
$\mf{sl}_2$, as defined by Etingof, Gan, and Ginzburg in \cite{EGG}.)
Define $H_{f,n} := (H_f)^{\otimes n} \rtimes S_n = S_n \wr H_f$. Then (we
work here over $k = \C$, and) the analysis above yields:

\begin{theorem}
$H_{f,n}$ has a triangular decomposition. Moreover, if $1 + f \neq 0$,
then $\calo$ is an abelian, finite length, self-dual category with block
decomposition, as above. Each block is a highest weight category.
\end{theorem}

\noindent This is because $A = H_f$ is then a strict Hopf RTA that
satisfies Condition (S3), from \cite{Kh}. Moreover, from \cite{KT},
$CC_A(\la)$ is finite for all $\la \in G$, since the center is ``large
enough"). We can also characterize all finite-dimensional simple modules
in $\calo$, by \cite[Theorem 11]{Kh}, and the values of $f$ for which
complete reducibility holds in $\calo$ (see \cite[Theorem 10.1]{GGK}).

\subsection{Digression: Deforming a smash product}

Suppose a Lie algebra $\mf{g}$ acts on a vector space $V_0$. One defines
a Lie algebra extension $V_0 \rtimes \mf{g}$, via:
\[ [v,v'] := 0 \qquad [X,v] := X(v)\ \forall v,v' \in V_0,\ X,X' \in
\mf{g} \]

Now suppose that $V,V' \subset V_0$ are finite-dimensional
$\mf{g}$-submodules, and we want to deform the relations $[v,v']$ (inside
the algebra $\mf{U}(V_0 \rtimes \mf{g})$), with desired images in
$\mf{Ug}$. In particular, $[X,[v,v']]$ should equal $\ad (X)([v,v']) \in
\mf{Ug}$. Thus we need a $\mf{g}$-invariant element $\omega$ of $\hhh_k(V
\wedge V',\mf{Ug})$, i.e., $\omega \in \hhh_{\mf{g}}(V \wedge V',\mf{Ug})
= ((V \wedge V')^* \otimes_k \mf{Ug})^{\mf{g}}$. In particular, if $V=V'$
has dimension 2, then $\omega \in (k \otimes_k \mf{Ug})^{\mf{g}} =
\mf{Z}(\mf{Ug})$, so $[v,v']$ must be central.

Now apply this to $H_{f,n}$, where $\mf{sl}_2$ acts on $V = kX \oplus kY$
inside each copy $(H_f)_i$. Hence any deformed relations $[Y_i,X_i]$ must
be of the form $f(\dd_i)$.
Moreover, the element $\sum_i [Y_i,X_i]$ must be central in the skew
group ring $\mf{Ug} \rtimes S_n$ (for $\mf{g} = \mf{sl}_2^{\oplus n}$),
for it is central in $\mf{Ug}$ as above, and it is invariant under any
transposition, hence under $S_n$.

\subsection{Certain other deformations do not preserve the triangular
decomposition}

We consider certain deformations of $H_{f,n}$ and show that they do not
preserve the triangular decomposition. More precisely, consider relations
of the form
\begin{equation}\label{E2}
[Y_i,X_j] =
\begin{cases}
	f(\dd_i) + \sum_{l \neq i} (c s_{il} + d)(m(1 \otimes S)
	\dd_{il})(\dd) & \text{if $i = j$}\\
	u s_{ij} + v s_{ij}(m(1 \otimes S) \dd_{ij})(\dd) +
	w_{ij} & \text{if $i \neq j$}\\
\end{cases}
\end{equation}
where

\noindent $c,d,u,v \in k$, and $w_{ij} \in \mf{Ug}$ for all $i \neq j$
(where $\mf{g} = \spl^{\oplus n}$),

\noindent $m$ is the multiplication map $: H_{f,n} \otimes H_{f,n} \to
H_{f,n}$,

\noindent $\dd_{ij}$ is the comultiplication in $\usl$, with image in
$\usl_i \otimes_k \usl_j \subset H_{f,n}$,

\noindent $S$ is the Hopf algebra antipode map on $\mf{U}(\spl^{\oplus
n})$, taking $X \in \spl^{\oplus n}$ to $-X$,

\noindent and $\dd$ is the Casimir element in $\usl$.

\noindent (Note that together with this subalgebra $\mf{Ug}$ comes its
{\em Cartan subalgebra} $\mf{h} = \oplus \mf{h}_i$, where $\mf{h}_i = k
\cdot H_i \subset (H_f)_i$. Moreover, $H_f^{\otimes n}$ is
$\mf{h}$-semisimple.)

We consider these relations because they are similar to those found in
certain wreath-product-type deformations (see \cite{EM}). Here, we find
symplectic reflections of the form $s_{ij} \gamma_i \gamma_j^{-1}$, which
equals $s_{ij} (m(1 \otimes S)\dd_{ij})(\gamma)$ under the Hopf algebra
structure (note that $\gamma_i = f_i(\gamma)$ here). However, we would
still like our deformations to have the triangular decomposition.

\begin{theorem}
The only deformations of the type in equation \eqref{E2}, that also have
the triangular decomposition, occur when $c,d,u,v,w_{ij}$ are all zero.
\end{theorem}

\begin{proof}
Check that
$(m(1 \otimes S)\dd_{ij})(\dd) = \dd_i + \dd_j + (e_i f_j + f_i e_j
+(h_i h_j) / 2))$.
%(This is easy.) 
Now note that $X_j$ and $Y_i$ are all weight vectors for $\ad \mf{h}$;
hence so also must be their commutator. Suppose $X_j \in
(H_{f,n})_{\eta_j}$ and $Y_i \in (H_{f,n})_{-\eta_i}$ for all $i,j$. Then
we should get $[Y_i,X_j] \in (H_{f,n})_{\eta_j - \eta_i}$ for all $i,j$.

Let us denote $e_i f_j + f_i e_j + (h_i h_j)/2$, by $m_{ij} = m_{ji}$. We
now check the relations for $i=j$, for we know that $\sum_i [Y_i,X_i]$
must be central in $\mf{Ug} \rtimes S_n$:
\[ \sum_i [Y_i,X_i] = \sum_i f(\dd_i) + \sum_{i \neq j} (c s_{ij} +
d)(\dd_i + \dd_j + m_{ij}) \]

\noindent The first term is obviously central. We now use that the
commutator of $e_k$ with this sum is zero (for fixed $k$), to show
$c=d=0$. This commutator equals (up to scalar multiples)
\[ \sum_{i \neq k} \bigg( [e_k,c s_{ik}](\dd_i + \dd_k + m_{ik}) + (c
s_{ik} + d)[e_k,\dd_i + \dd_k + m_{ik}] \bigg) \]

\noindent The first term equals $cs_{ik}(e_i-e_k)(\dd_i + \dd_k +
m_{ik})$ and the second term is $(c s_{ik} + d)[e_k,m_{ik}]$, so we get
that $[e_k, \sum_i [Y_i,X_i]]$ equals
\[ = \sum_{i \neq k} cs_{ik}(e_i - e_k)(\dd_i + \dd_k + m_{ik}) + (c
s_{ik} + d)(e_i h_k - h_i e_k) \]

\noindent To satisfy the triangular decomposition, we must have $c=0$
(consider the coefficient of $s_{ik} f_i e_i^2$, in $c s_{ik} e_i \dd_i =
c s_{ik} \dd_i e_i$), and hence $d=0$.

Next, when $i \neq j$, we need $[Y_i,X_j]$ to be an $\ad \mf{h}$-weight
vector of weight $\eta_j - \eta_i$. But $w_{ij} \in \usln$, and
$\usln_{\eta_j - \eta_i} = 0$ if $i \neq j$, since each weight space has
weight in $\oplus_i \Z (2 \eta_i)$.
Similarly, the $\ad h_i$-actions on $[Y_i,X_j]$ and the other terms do
not agree:
\[ [h_i, s_{ij}(u + v m_{ij})] = s_{ij}(h_j - h_i)(u + v m_{ij}) + v
s_{ij}[h_i,m_{ij}] \]

\noindent and thus the LHS is not a weight vector, unless $u = v = 0$ as
well.
\end{proof}

\subsection*{Acknowledgements}
I thank Mitya Boyarchenko, Vyjayanthi Chari, Victor Ginzburg, Nicolas
Guay, and Akaki Tikaradze, for suggestions and valuable discussions.

\end{document}